\documentclass[a4paper]{article}

\usepackage[utf8]{inputenc}
\usepackage[english]{babel}
\usepackage{algorithm}
\usepackage{algorithmic}

\usepackage{graphicx}
\usepackage{amsfonts}
\usepackage{amsmath}
\numberwithin{equation}{section}
\usepackage{amssymb}
\usepackage{mathtools}
\usepackage[table,usenames,dvipsnames,svgnames,x11names]{xcolor}
\usepackage{subcaption} 
\usepackage{physics}  
\usepackage[section]{placeins} 
\usepackage{calc} 
\usepackage{amsthm} 
\theoremstyle{definition} 
\newtheorem{definition}{Definition}[section] 

\theoremstyle{plain}
\newtheorem{theorem}{Theorem}[section]
\newtheorem{proposition}{Proposition}[section]
\newtheorem{corollary}{Corollary}[theorem]
\newtheorem{lemma}[theorem]{Lemma}
\theoremstyle{remark}

\usepackage{hyperref}

\usepackage[style=numeric, giveninits]{biblatex} 
\usepackage{csquotes}

\usepackage{geometry}
\usepackage{mathtools}

\renewcommand{\Rn}[1]{\mathbb{R}^{#1}}
\newcommand{\Rnm}[2]{\mathbb{R}^{#1 \times #2}}

\newcommand{\C}{\mathbb{C}}

\newcommand{\Cnm}[2]{\mathbb{C}^{#1 \times #2}}

\newcommand{\K}{\mathcal{K}_\sa} 
 
\newcommand{\V}{V} 
\newcommand{\Vplus}{\Tilde{\V}} 
\newcommand{\VOm}{S}

\newcommand{\ps}[2]{\langle #1 , #2 \rangle}
\newcommand{\psOm}[2]{\langle \Om #1 , \Om #2 \rangle }

\newcommand{\Om}{\Omega} 
\newcommand{\POm}{P_{\K}^{\Om}}
\newcommand{\Pkr}{P_{\K}}
\newcommand{\A}{A} 
\newcommand{\Hsa}{H_\sa} 
\newcommand{\Hba}{H_{\ba}} 
\newcommand{\Hbap}{\Hba^+} 
\newcommand{\Hsaplus}{\Tilde{\Hsa}} 
\newcommand{\eivec}{u} 
\newcommand{\eival}{\lambda} 

\newcommand{\rivec}{\tilde{\eivec}} 
\newcommand{\yrivec}{{y}} 
\newcommand{\rival}{\tilde{\eival}} 

\newcommand{\vinit}{v_1} 
\newcommand{\vinitj}{v_1^+} 
\newcommand{\rsa}{r_\sa} 
\newcommand{\rsaplus}{\Tilde{r}_\sa}
\newcommand{\rba}{r_{\ba}}

\newcommand{\Q}{Q} 
\newcommand{\shi}{\mu} 

\newcommand{\sa}{k} 
\newcommand{\nshi}{p} 
\newcommand{\ba}{\sa + \nshi} 
\newcommand{\dsk}{d} 

\newcommand{\polp}{\psi_{\nshi}}

\newcommand{\Omperp}{\perp^{\Om}}
\DeclareMathOperator{\Sp}{span}

\newcommand{\Span}[1]{\Sp \{ #1 \} }
\renewcommand{\epsilon}{\varepsilon}

\newcommand{\rArno}{randomized Arnoldi}
\newcommand{\cpol}[2]{p_{#1}(#2)}

\title{Randomized Implicitly Restarted Arnoldi method for the non-symmetric eigenvalue problem}

\author{Jean-Guillaume de Damas\thanks{Sorbonne Universite, Inria, Universite de Paris Laboratoire Jacques-Louis Lions, Paris, France}, Laura Grigori\thanks{Laboratory for Simulation and Modelling, Paul Scherrer Institute, Switzerland; Institute of Mathematics, EPFL, Switzerland}}

\date{\today}

\renewenvironment{abstract}
 {\par\textbf{\abstractname.}\ \ignorespaces}
 {\par\medskip}
\newenvironment{keywords}
 {\par\textbf{Keywords.}\ \ignorespaces}
 {\par\medskip}
\newenvironment{AMS}
 {\par\textbf{AMS subject classification.}\ \ignorespaces}
 {\par\medskip}

\AtBeginBibliography{\small}

\begin{document}

\maketitle

\begin{abstract}
    In this paper, we introduce a randomized algorithm for solving the non-symmetric eigenvalue problem, referred to as randomized Implicitly Restarted Arnoldi (rIRA).  This method relies on using a sketch-orthogonal basis during the Arnoldi process while maintaining the Arnoldi relation and exploiting a restarting scheme to focus on a specific part of the spectrum.
    We analyze this method and show that it retains useful properties of the Implicitly Restarted Arnoldi (IRA) method, such as restarting without adding errors to the Ritz pairs and  implicitly applying polynomial filtering. Experiments are presented to validate the numerical efficiency of the proposed randomized eigenvalue solver.
\end{abstract}



\section{Introduction}
The problem of finding the eigenvalues and/or eigenvectors of a given linear transformation arises in many applications. It consists in finding a set of eigenpairs $(\eivec, \eival)$ of a matrix $\A \in \Rnm{n}{n}$, with $\eivec \in \Rn{n}$, $\norm*{\eivec} = 1$, and $\eival \in \C$ such that
\begin{equation}
    \A \eivec = \eival \eivec.
\end{equation}
This paper focuses on the computation of a small subset of eigenpairs of a matrix $\A$ which may be non-symmetric, large and sparse. A well-known tool to extract certain eigenpairs of $\A$ is the Rayleigh-Ritz method, see a detailed analysis in \cite{saad2011numerical}. It consists of a projection in a $\sa$ dimensional subspace $\K$ with $\sa \ll n$, whose capacity to accurately approximate eigenpairs depends on the subspace $\K$ and the projector.
Concerning the subspace, it is relevant to use Krylov subspaces defined as
$$\K(\A,\vinit) = \Span{\vinit,\A \vinit,\A^2 \vinit,\dots,\A^{\sa-1} \vinit},$$
for a given starting vector $\vinit$. They are simple to obtain provided that matrix-vector products are affordable, and they contain desirable spectral information of $\A$. Regarding the projection, the Arnoldi's method, also presented in \cite{saad2011numerical}, performs an orthogonal projection into $\K$ which is well suited for stable computations in numerical linear algebra. This technique requires the calculation of an orthogonal basis for $\K$ and its storage, which becomes expensive in memory and floating point operations (flops) as $\sa$ increases. To alleviate this problem, different restarting schemes have been proposed in the literature. One of them is the Implicitly Restarted Arnoldi method (IRA), introduced by D. Sorensen in 1992 in \cite{Sorensen1992ImplicitApplicationPolynomial} and extended by R. Lehoucq's thesis \cite{Lehoucq1995DeflationTechniquesImplicitly}. A key feature of IRA is its ability to focus on a specific part of the spectrum without explicitly modifying $\A$. The motivating result for this method is that an invariant subspace of dimension $\sa$ of $\A$ can be constructed using an iterative procedure that starts using information from the Schur decomposition of $\A$. That is from an orthogonal $Q \in \Rnm{n}{\sa}$ satisfying $AQ = QR$, where $R$ is upper-triangular with eigenvalues of $\A$ on its diagonal. The iterative restart framework, called outer iteration, is used to drive the starting vectors of successive Arnoldi factorizations, called inner iteration, into the range of this a priori unknown $Q$. If an invariant subspace of $\A$ is obtained, exact eigenpairs can be computed from it. The cost of the orthogonalization process can be reduced by using randomization. Randomization, also known as sketching, is a dimensionality reduction technique that relies on the Johnson-Lindenstrauss (JL) lemma which states that the distances between $\sa$ vectors of $\Rn{n}$ can be preserved in a $\dsk$-dimensional subspace within a $(1 \pm \epsilon)$ factor, where $\dsk \approx \log \sa / \epsilon^2$, see \cite{Venkatasubramanian2011JohnsonLindenstraussTransform}. When the subspace spanned by the $\sa$ vectors is not known in advance, \emph{oblivious subspace embeddings} can be used.  These embeddings allow to satisfy the JL lemma for any subspace of dimension $\sa$ with high probability. It is proven in \cite{WoodruffComputationalAdvertisingTechniques2014} that matrices drawn from the Gaussian distribution are oblivious subspace embeddings when $\dsk \approx \sa / \epsilon^2$, but other embeddings from various distributions can be considered as notably reviewed in \cite{MartinssonRandomizednumericallinear2020}[Sections 8,9].

Randomization was successfully used to derive fast and accurate methods for solving least squares problems, computing the low rank approximation of a matrix or selecting a set of columns, for a detailed description see \cite{HalkoFindingStructureRandomness2011, WoodruffComputationalAdvertisingTechniques2014, Murray2022RandomizedNumericalLinear} for example. It was used in the recent years in the context of Krylov subspace methods, in particular for solving linear systems of equations \cite{BalabanovRandomizedGramSchmidt2022, NakatsukasaFastAccurate2021, Balabanov2021RandomizedblockGram, Timsit2023RandomizedOrthogonalProjection}. Randomized Arnoldi introduced in \cite{BalabanovRandomizedGramSchmidt2022} relies on a randomized orthogonalization process that produces a well conditioned basis of the Krylov subspace and thus can be efficiently used in a randomized version of GMRES.  
A different approach consists in using sketching independently of the construction of the Krylov basis, as in sketched GMRES \cite{NakatsukasaFastAccurate2021}, or sketch and select \cite{Guettel2023sketchselectArnoldi} that aims at building a well conditioned Krylov basis by appealing to the best subset selection problem.  Restarting in the context of linear systems is discussed in \cite{Burke2023GMRESrandomizedsketching,Jang2024RandomizedflexibleGMRES}. Randomization can be also used for computing functions of matrices, see \cite{Guettel2023RandomizedSketchingKrylov,Cortinovis2022SpeedingKrylovsubspace,Palitta2023Sketchedtruncatedpolynomial}. 

Few results exist for the eigenvalue problem due its complexity.  In the broader context of the generalized eigenvalue problem $\A \eivec = \eival B \eivec$, \cite{Saibaba2015Randomizedalgorithmsgeneralized} relies on randomized low rank approximation for the Hermitian case,  while \cite{Kalantzis2021FastRandomizedNon} exploits the randomized range-finder algorithm to extract eigenvalues situated in a disk for the pencil $(A,B)$ in the non-Hermitian case.  Both use orthogonal bases. Randomness is used in other algorithms such as in randomized FEAST in \cite{Yin2016randomizedFEASTalgorithm}, where FEAST is extended to the non-Hermitian case by using a random subspace instead of a Krylov one. In the nonlinear case, the randomization of AAA from \cite{NakatsukasaAAAAlgorithmRational2018} is investigated in \cite{GuettelRandomizedsketchingnonlinear2022}.  A randomized version of Rayleigh-Ritz is briefly introduced in \cite{Balabanov2021RandomizedblockGram, NakatsukasaFastAccurate2021}. 

In this article, we introduce the randomized Implicitly Restarted Arnoldi (rIRA) method, that relies on a sketched orthonormal basis and a restarting scheme that allows to seek a specific subset of eigenpairs of a non-symmetric matrix $\A$. It relies on the \rArno{} factorization \cite{BalabanovRandomizedGramSchmidt2022} of the matrix $\A$,
\begin{equation}
    \A \V = \V \Hsa + \rsa e_k^T, 
\end{equation}
where $\V$ is an $\Om$-orthonormal (also called sketched
orthonormal) basis of the Krylov subspace $\K$. This means that $(\Om \V)^T (\Om \V) = I_\sa$, where $\Om$ is a  subspace embedding for $\K$. If $\tilde{Q}$ is the $\Om$-orthonormal factor of the $\Om$-orthonormal Schur decomposition of $A$, that is it satisfies 
\begin{equation}
    \A \tilde{Q} = \tilde{Q} T
\end{equation}
with $T$ upper triangular, then we show that $\rsa$ is zero if and only if the starting vector $\V(:,1) = \vinit$ lies in the subspace spanned by $\tilde{Q}$ (see Theorem~\ref{th:schur-invariant}). 
In this situation, $\V$ is an invariant subspace of $\A$ and we can extract exact eigenpairs from it. We show that the update of a \rArno{} factorization using the deterministic shifted QR algorithm results into a legitimate \rArno{} factorization, and this leads to an iterative restarting process. The convergence can be monitored using sketches of the residuals that are easily computable and  approximate well the residual error norms of the form $\norm*{\A \rivec - \rival \rivec} / \norm*{\rivec}$. 

We provide a theoretical analysis that shows that some of the results defining the convergence behavior of IRA hold for randomized IRA, up to a factor of $1 + O(\epsilon)$ and with high probability.
We start by considering \rArno{} as an oblique projection $\POm$ onto $\K$, following the work of \cite{Balabanov2021RandomizedblockGram} where it is shown that randomized Rayleigh-Ritz solves the eigenvalue problem $\POm \A \POm \rivec = \rival \rivec$. We then derive bounds on the distance from an eigenvector $u$ to the Krylov subspace through the oblique projection $\POm$, that is $\norm*{(I - \POm)u}$, by relying on bounds which use an orthogonal projection, derived in \cite{saad2011numerical,Bellalij2016distanceeigenvectorKrylov,Bellalij2010FurtherAnalysisArnoldi}.
We show an optimality property satisfied by the characteristic polynomial $\hat{p}_{\sa}$ of $\Hsa$ (see Theorem~\ref{th:charac-poly-min}), which characterizes the randomized oblique projection as a minimizer of $\norm*{\Om p(\A) \vinit}$ among the set $\mathcal{PM}_\sa$ of all monic polynomials $p$ of degree $\sa$, that is
\begin{equation}
    \hat{p}_{\sa} = \arg \min_{p \in \mathcal{PM}_\sa } \norm*{\Om p(\A) \vinit}.
\end{equation}
This result is used to prove convergence of rIRA for a specific shifts selection strategy.
We also show that, similarly to  \cite{Sorensen1992ImplicitApplicationPolynomial}, a \rArno{} factorization is uniquely defined by its starting vector in the randomized implicit Q theorem (see Theorem~\ref{th:r-implicit-Q}),
\begin{equation}
    \left\{ \begin{array}{c}
        AV = VH + re_k^T \\
        AQ = QG + fe_k^T 
    \end{array} \right.
    \; \text{are equal if their starting vectors are equal, i.e.\@ if} \; V(:,1) = Q(:,1).
\end{equation}
We also prove that polynomial filtering is implicitly applied when restarting, that is if $\Vplus$ is the new Krylov basis after the update of the \rArno{} factorization, then its associated Krylov subspace is
\begin{equation}
    \Span{\Vplus} = \K(\A, \polp(\A) \vinit)
\end{equation}
where $\polp(\A)$ is a polynomial that annihilates the coordinates of $\vinit$ in unwanted directions. Moreover, we show that no error is added to the Ritz pairs when restarting.

The paper is organized as follows. Section~\ref{sec:preliminaries} introduces notations and several definitions and results used throughout the paper. In section~\ref{sec:orthog} we introduce an algorithm for computing a sketched orthonormal basis of a Krylov subspace, referred to as \emph{randomized Classical Gram-Schmidt 2} (rCGS2) and we compare it with other randomized orthogonalization processes for an ill-conditioned case.  We then recall the \rArno{} procedure defined in \cite{NakatsukasaFastAccurate2021,Balabanov2021RandomizedblockGram} and its usage in randomized Rayleigh-Ritz.
The main algorithm is presented in section~\ref{sec:rIRA} along with justification on why we recover a \rArno{} factorization after the update step. We also detail convergence criterion and how to cheaply update sketched quantities. Theoretical results of \rArno{} and rIRA are derived in section~\ref{sec:analysis}, notably on the convergence of the \rArno{} method, the uniqueness of a \rArno{} factorization, and its usage in the restart procedure of rIRA. Section~\ref{sec:numericals} studies the efficiency of rIRA by using its implementation in the Julia programming language. We start by outlining the capacity of rIRA to target a specific part of the spectrum, by focusing on the smallest or largest modulus eigenvalues. We then show that rIRA reaches the same accuracy as the ARPACK implementation of IRA, while requiring less time per iteration for different matrices considered in our test set.

\section{Preliminaries}
\label{sec:preliminaries}
\subsection{Notations}
We consider the usual $\Rn{n}$ vector space. The inner product $\ps{.}{.}$ is $\ps{x}{y} \coloneqq x^T y$. The induced norm is the 2-norm, denoted here by $\norm*{x} \coloneqq \norm*{x}_2 = \sqrt{\ps{x}{x}}$, unless stated otherwise. A unit vector $v$ is such that $\norm*{v} = 1$. If a vector $v$ has all its coordinates zero, we use $v = 0 \in \Rn{n} \iff \norm*{v} = 0 \in \mathbb{R}$.  Two vectors $x,y$ are orthogonal when $\ps{x}{y} = 0$ and a matrix $Q \in \Rnm{n}{m}$ is orthonormal when $Q^T Q = I_m$, where $I_m$ is the square identity matrix of dimension $m$. Sub-matrices are noted $H(1:i,1:j)$ for the $i$ first rows and $j$ first columns. Indexes are omitted when all rows (or columns) are selected such as in $H(:,i)$ which selects the i'th column of $H$. Entries are noted $h_{i,j} \coloneqq H(i,j)$. Given $M \in \Rnm{n}{\sa}$, $\Span{M}$ denotes the column subspace $\mathcal{M} = {\Span{M(:,1),\dots,M(:,\sa)}}$.

\subsection{The Rayleigh-Ritz procedure and the Arnoldi algorithm}

Given $\A \in \Rnm{n}{n}$ and some $i \in \mathcal{I} \subset \{1,\dots,n \}$, we consider the problem of computing the eigenpairs $(\eivec_i, \eival_i)$, where
\begin{equation}
    \A \eivec_i = \eival_i \eivec_i, \quad \eivec_i \in \Rn{n}, \; \norm{\eivec_i} = 1, \; \eival_i \in \C,
    \label{eq:EVP}
\end{equation}
When $n$ is large, a common technique is to seek an approximation $\rivec_i$ in a smaller $\sa$-dimensional subspace $\K$, with $\sa \ll n$. The $\sa$ degrees of freedom arising from this choice are fixed by imposing the Petrov-Galerkin condition $\A \rivec_i - \rival_i \rivec_i \perp \K $. 
This leads to the Rayleigh-Ritz procedure with orthogonal projections. One of the most popular choices is to use a Krylov subspace, defined from an initial vector $\vinit \in \Rn{n}$ as
\begin{equation}
   \K \coloneqq \Span{\vinit, \A \vinit, \dots, \A^{\sa-1}\vinit}, 
\end{equation}
see \cite[Chapters 4 and 6]{saad2011numerical} for more details. For numerical stability, the Arnoldi iteration constructs an orthonormal basis of $\K$ by using Gram-Schmidt, and produces the factorization
\begin{equation}
    \A \V = \V \Hsa + h_{\sa+1,\sa} v_{\sa+1} e_\sa^T,
    \label{eq:arnoldiRelation}
\end{equation}
where $\V \in \Rnm{n}{k}$ is an orthonormal basis of $\K$, $v_{\sa+1} \in \Rn{n}$ is a unit vector orthogonal to $\V$, and $\Hsa \in \Rnm{k}{k}$ is an upper Hessenberg matrix, that is upper triangular with non zero elements on the first subdiagonal. Indeed the elements $h_{i+1,i}$ correspond to the norm of $v_i$ after orthogonalization against the previous $v_j$'s and before normalization. They are thus positive, and when all of them are non zero, i.e.\@ $h_{i+1,i} > 0, \; i=1,\dots,\sa-1$, we say that $\Hsa$ is unreduced. It is equivalent to the fact that $\dim(\K) = \sa$ since all $v_i$'s are non zeros and orthogonal to each other. In the following we define $\rsa \coloneqq h_{\sa+1,\sa} v_{\sa+1}$ for brevity. We summarize this discussion in the following definition.
\begin{definition}[Rayleigh-Ritz approximation for eigenpairs using Arnoldi iteration]
    \label{eq:rrEVP}
    Given a unit vector $\vinit \in \Rn{n}$, let $\K = \Span{\vinit, \A \vinit, \dots, \A^{\sa-1}\vinit} \subset \Rn{n}$ of dimension $\sa$. The resulting Arnoldi factorization is 
    \begin{equation}
        \label{eq:arnoDef}
       \A \V = \V \Hsa + \rsa e_\sa^T,  
    \end{equation}
    with $\V \in \Rnm{n}{k}$ an orthonormal basis for $\K$ and  $\V(:,1) = \vinit$. Given the eigenpairs $(\yrivec_i, \rival_i)$ of $\Hsa \in \Rnm{k}{k}$, the Ritz pairs of $\A$ are $(\rivec_i \coloneqq \V \yrivec_i, \rival_i)$ such that 
    \begin{equation}
        \A \rivec_i - \rival_i \rivec_i \perp \K \; \text{with} \; \rivec_i \in \K, \; \norm{\rivec_i} = 1, \; \rival_i \in \C.
    \end{equation}
\end{definition}
By multiplying Equation~\eqref{eq:arnoDef} with $\yrivec_i$, we obtain
$$\A (\V \yrivec_i) - \lambda_i (\V \yrivec_i) = \rsa e_\sa^T \yrivec_i. $$ 
This relation reflects the errors made when approximating the eigenpairs of $\A$ by the Ritz pairs and shows the close relation between $\rsa$ and the residual $\A \rivec_i - \lambda_i \rivec_i$. The happy break down of the Arnoldi process happens when $h_{\sa+1,\sa} = \norm*{\rsa} = 0$, and thus 
$$\A \V = \V \Hsa. $$
This means that an invariant subspace of $\A$ was found while computing $\K$ starting from $\vinit$, with $\V$ being a basis for this invariant subspace. In this situation, an eigenpair $(\rival_i, \yrivec_i)$ of $\Hsa$ corresponds to an exact eigenpair $(\rival_i, \V \yrivec_i)$ of $\A$,
$$ \A \V \yrivec_i = \V \Hsa \yrivec_i = \V \rival_i \yrivec_i \iff \A \rivec_i = \rival \rivec_i .$$
Otherwise, when $h_{\sa +1,\sa} \neq 0$, or equivalently $\rsa \neq 0$, the residual error made on a Ritz pair $(\rival_i, \rivec_i)$ of $\A$ satisfies the following relation available at no significant extra cost :
\begin{equation}
    \norm{\A \rivec_i - \rival_i \rivec_i}  = \norm{h_{\sa+1,\sa} v_{\sa+1} e_\sa^T \yrivec_i} = \norm{\rsa} \abs{e_\sa^T \yrivec_i} = h_{\sa+1,\sa} \abs{e_\sa^T \yrivec_i}.
    \label{eq:residualNorm}
\end{equation}

\subsection{The implicitly restarted Arnoldi method }

Restarting an Arnoldi process is a technique that seeks to build a new factorization by using a starting vector $\vinitj$ that uses information from a previous decomposition of size $\sa$ started from $\vinit$. Given that Arnoldi becomes expensive for large $n$ and $\sa$, in terms of memory to store the basis vectors or in terms of computation for the orthonormalization process, restarting the Arnoldi process is important. Efforts focused on computing a new $\vinitj$ such that $\K(\A,\vinitj)$ is more likely to be an invariant subspace for $\A$. An approach consists in constructing $\vinitj$ as a linear combination of the previous Ritz vectors $(\rivec_1,\dots,\rivec_\sa)$. Indeed, given $\sa$ exact eigenvectors $(\eivec_1,\dots,\eivec_\sa)$ of $\A$, performing a second Arnoldi process for $\sa$ steps starting with $\vinitj \coloneqq \sum_{i = 1}^{\sa}\alpha_i \eivec_i$ for some $\alpha_i \in \mathbb{R}$ gives in exact arithmetic
$$\Span{\vinitj,\dots,\A^{\sa-1}\vinitj} = \Span{\eivec_1,\dots,\eivec_\sa},$$ 
and $\K(\A,\vinitj)$ is an invariant subspace. Since the Ritz vectors approximate the exact eigenvectors and are available, it seems appropriate to use them. However, R. Morgan shows in \cite[Theorem 1]{Morgan1996restartingArnoldimethod} in a counter example that using $\vinitj \coloneqq \rivec_1 + \alpha \rivec_2$ with an arbitrary $\alpha$ produces a subspace $\Span{\vinitj,\A \vinitj}$ containing a new Ritz vector $\rivec_1^+$ that has a larger error as an approximation of the eigenvector sought $\eivec_1$ than the initial Ritz vector $\rivec_1$. It is subsequently stated in \cite{Morgan1996restartingArnoldimethod} that the only choice of $\alpha$ that does not increase the error in the approximation is $\alpha = - \frac{\beta_{\sa1}}{\beta_{\sa2}}$, where $\beta_{\sa i} \coloneqq h_{\sa+1,\sa} \abs*{e_k^T\yrivec_{i}}$ from Equation~\eqref{eq:residualNorm},  i.e.\@ it weights the linear combination by the residual errors committed on each Ritz pair. This generalizes to a combination of $\sa$ Ritz vectors and is then called explicit restarting.  An equivalent yet more stable method is the implicitly restarted Arnoldi algorithm (IRA) introduced by D. Sorensen in \cite{Sorensen1992ImplicitApplicationPolynomial}. R. Morgan proves in \cite[Theorem 2]{Morgan1996restartingArnoldimethod} that IRA is equivalent to computing the exact $\alpha_i$'s so that no error is added at each restart.

The original idea of Sorensen in \cite{Sorensen1992ImplicitApplicationPolynomial} is to update an Arnoldi factorization that started from $\vinit$ to implicitly modify the starting vector into a new $\vinitj$ without needing to  fully compute the factorization iteratively again. This $\vinitj$ has coordinates that are zero in unwanted directions, which makes $\K(\A,\vinitj)$  a better subspace to extract a set of wanted eigenpairs from $\A$.
The following discussion specifies this procedure. We state a result defining the real Schur decomposition of $\A$, for which the proof can be found in \cite{Golub2013Matrixcomputations}[Chapter 7].

\begin{proposition}[real Schur decomposition]
    If $A \in \Rnm{n}{n}$ then there exists an orthonormal $Q \in \Rnm{n}{n}$ such that 
    \begin{equation}
        AQ = QR 
        \label{eq:SchurDecompo}
    \end{equation}
    where $R \in \Rnm{n}{n}$ is block upper triangular and has the same eigenvalues as $\A$. Its diagonal blocks are of size $1 \times 1$ or $2 \times 2$, the latter accounting for complex conjugate pairs of eigenvalues. The columns of $Q$ are called Schur vectors. They can be chosen so that the eigenvalues $\eival_i$ of $\A$ appear in any order along the diagonal of $R$.
\end{proposition}

In the remaining of this paper, block upper triangular in the context of a real Schur factorization is also referred to as upper triangular. This theorem shows that every matrix $\A$ is orthonormally similar to an upper triangular one that has the eigenvalues of $\A$ on the diagonal. Moreover, note that if $q_\sa$ is the $\sa$-th column of $Q$, then
\begin{equation}
   \A q_\sa = \eival_\sa q_\sa + \sum_{i=1}^{\sa-1} r_{i,k} q_i. 
\end{equation}
Thus $\Span{q_1,\dots,q_\sa}$ is an invariant subspace of $\A$. 
This allows to define a $\sa$ partial Schur factorization $AQ = QR$, with $Q \in \Rnm{n}{\sa}, Q^TQ = I_\sa$, and $R \in \Rnm{\sa}{\sa}$ upper triangular. Given an arbitrary subset of eigenvalues $\{\eival_1,\dots,\eival_\sa \}$, this factorization is always obtainable by taking the first $\sa$ Schur vectors from Equation~\eqref{eq:SchurDecompo} where these Schur vectors are chosen to place $\{\eival_1,\dots,\eival_\sa \}$ on the top-left part of the $n \times n$ upper triangular matrix. Sorensen then shows a main motivational result.
\begin{theorem}[from Theorem 2.9 of \cite{Sorensen1992ImplicitApplicationPolynomial}]
    \label{th:sorensenMain}
    Let $\A \V = \V \Hsa +   \rsa e_\sa^T$ be a $\sa$-step Arnoldi factorization, with $\Hsa$ unreduced, i.e.\@ $h_{i+1,i} > 0, \; i=1,\dots,\sa-1$. Then $\rsa = 0$ and the columns of $\V$ span an invariant subspace of $\A$ if and only if  $\vinit = Qy$, where $\A Q = QR$ is a $\sa$ partial Schur factorization of $\A$ with $Q \in \Rnm{n}{\sa}$ and for some $y \in \Rn{\sa}$.
\end{theorem}
In other words, if for a given factorization $\vinit$ is a linear combination of $\sa$ Schur vectors of $\A$, then there is a happy break down in the Arnoldi process at step $\sa$.  Hence $\rsa = 0$ and exact eigenpairs can be extracted using $\Hsa$ and $\V$, according to Equation~\eqref{eq:residualNorm}.
An Arnoldi factorization can be updated by using $\nshi$ steps of the shifted QR algorithm on the upper Hessenberg matrix $\Hba$, as described in \cite{Sorensen1992ImplicitApplicationPolynomial}. The shifted QR algorithm is notably reviewed in \cite{Golub2013Matrixcomputations}. The $\nshi$ shifts can be obtained from the eigenvalues of $\Hba$. It is shown in \cite[Theorem 4.4]{Lehoucq1995DeflationTechniquesImplicitly} that given this specific shifts selection in IRA, $\vinitj$ is a linear combination of $\sa$ approximate Schur vectors of $\A$ and no additional error is added through the restart. The approximate Schur vectors are the columns of a matrix $\Tilde{Q}$ satisfying 
\begin{equation}
    \label{eq:approxSchurVecs}
    \A \Tilde{Q} = \Tilde{Q} \Tilde{R} +  \alpha \tilde{\rsa} e_\sa^T,
\end{equation}
$\Tilde{R}$ being upper triangular and $\alpha$ a scalar. In the end, IRA relies on an iterative procedure where Arnoldi factorizations are computed, updated by the shifted QR algorithm, truncated to maintain an Arnoldi structure, and then expanded again. Algorithm~\ref{alg:IRA}, also presented in \cite[Chapter 7]{saad2011numerical}, describes this process. Further details are given later in the randomized setting. 
\begin{algorithm}[htbp]
    \caption{Implicitly restarted Arnoldi algorithm (IRA)}
    \label{alg:IRA}
    \begin{algorithmic}[1]
        \STATE Perform $\sa$ steps of the Arnoldi procedure to obtain 
        $\A \V = \V \Hsa + \rsa e_{\sa}^T$ \label{alg:IRA-arnoldi-Init}
        \WHILE{convergence not declared} 
        \STATE Extend the Arnoldi factorization to a $\ba$ one through $\nshi$ additional steps: $\A \V = \V \Hba + \rba e_{\ba}^T$ \label{alg:IRA-arnoldi-extend}
        \STATE Compute the eigenvalues $(\rival_1, \dots, \rival_{\ba})$ of $\Hba$. Either declare convergence using some of these and stop, or select $\nshi$ shifts $(\shi_1,\dots,\shi_\nshi)$ among them. 
        \label{alg:IRA-get-shifts}
        \STATE // Initialize $\Hbap$ to $\Hba $ and perform $\nshi$ shifted QR steps on it:
        \FOR{$i = 1,\dots, \nshi$} \label{alg:IRA-qr-start}
        \STATE $(\Hbap - \shi_i I) = Q^i R^{i}$
        \STATE $\Hbap = (Q^{i})^T \Hbap Q^i$ 
        \ENDFOR \label{alg:IRA-qr-end}
        \STATE Set $\Q = Q^1 \dots Q^{\nshi}$ and $\Vplus = \V \Q$ \label{alg:IRA-update}
        \STATE Set $\rsaplus = \Hbap(\sa+1,\sa) \cdot \Vplus(:,k+1) + \Q(\ba,\sa) \cdot \rba$
        \STATE Truncate $\Hsaplus = \Hbap(1:\sa,1:\sa)$ and $\Vplus = \Vplus(:,1:\sa)$
        \STATE // Continue with the resulting Arnoldi factorization $\A \Vplus = \Vplus \Hsaplus + \rsaplus e_{\sa}^T$
        \ENDWHILE  
    \end{algorithmic}
\end{algorithm}

\subsection{Sketching }

We introduce several definitions and results mainly from \cite{WoodruffComputationalAdvertisingTechniques2014,MartinssonRandomizednumericallinear2020}. 
\begin{definition}[$\epsilon$-embedding]
    $\Om \in \Rnm{\dsk}{n}$ is an $\epsilon$-embedding for some $\sa$-dimensional subspace $\K \subset \Rn{n}$ if 
    \begin{equation}
       \forall \; x,y \in \K, \; \abs{\ps{\Om x}{\Om y}-\ps{x}{y}} \leq \epsilon \norm*{x}\norm*{y}, 
    \end{equation}
    or equivalently 
    \begin{equation}
        \forall \; x \in \K, \; (1-\epsilon)\norm*{x}^2 \leq \norm*{\Om x}^2 \leq (1+\epsilon)\norm*{x}^2.
        \label{eq:epsembedd}
    \end{equation}
\end{definition} 
The following definition allows to obtain such embeddings for subspaces that are not known in advance.
\begin{definition}[Oblivious subspace embedding (OSE)]
    \label{def:OSE}
    $\Om \in \Rnm{\dsk}{n}$ is an $(\epsilon,\delta,\sa)$ oblivious subspace embedding if it is an $\epsilon$-embedding for any subspace $\K$ of dimension $\sa$ with probability $1-\delta$.
\end{definition}
We present three different OSEs, other choices can be found for example in \cite{MartinssonRandomizednumericallinear2020}[Chapter 8 and 9]:
\begin{itemize}
    \item Gaussian, where $\Om = \frac{1}{\sqrt{\dsk}} G$, each entry of $G$ follows the standard normal distribution $\mathcal{N}(0,1)$. It is an OSE when $\dsk = O(\epsilon^{-2}(\sa + log \frac{1}{\delta}))$. While well suited for parallel architectures, it is expensive to apply to a large vector with a cost of $O(n \dsk)$ floating point operations (or flops).  
    \item Sumbsampled Randomized Hadamard Tranform (SRHT), where $\Om =  \sqrt{\frac{n}{\dsk}}PHD$, with $D \in \Rnm{n}{n}$ a diagonal matrix of random signs, $H \in \Rnm{n}{n}$ a Hadamard matrix assuming $n$ is a power of $2$, and $P \in \Rnm{\dsk}{n}$ an uniform sampling matrix that picks $\dsk$ columns from $n$ at random. It is an OSE when $\dsk = O(\epsilon^{-2}(\sa + log \frac{n}{\delta})\log \frac{\sa}{\delta})$. It costs $O(n \log \dsk)$ flops to apply it to a vector, thus it is more efficient than Gaussian, and its block variant is suitable for parallel architectures, see \cite{BalabanovBlocksubsampledrandomized2022a}.
    \item Sparse Sign matrix of parameter $\zeta$: $\Om = \frac{1}{\sqrt{\zeta}}  [s_1 \dots s_n]$, where each $s_i$ is a sparse column with exactly $\zeta$ random signs $\pm 1$ draw with probability $1/2$. It requires to store $O(\zeta n \log \dsk)$ numbers and costs $O(\zeta n)$ flops to apply to a vector. In practice, this is faster than the SRHT for $\zeta = 8$ in most implementation as advised in \cite{MartinssonRandomizednumericallinear2020} since the theoretical bound of $O(n \log \dsk)$ flops for SRHT is difficult to attain. It requires a high-performance sparse library for an efficient implementation. It has been theoretically shown to be an OSE for $\dsk = O(\epsilon^{-2}(\sa \log \sa))$ and $\zeta = O(\epsilon^{-1}(\log \sa))$ in \cite{Cohen2015NearlyTightOblivious}.
\end{itemize}
Sketching is of particular interest when $\sa \ll n$, i.e.\@ it is applied to a few  high dimensional vectors. A usual choice is then $\epsilon = 1/2$ that allows to preserve inner products while satisfying $\sa < \dsk \ll n$. In the numerical experiments we take $\dsk$ to be a small multiple of $\sa$, for example $4 \sa$. We further introduce the following definitions:
\begin{itemize}
    \item Two vectors $x,y \in \Rn{n}$ are $\Om$-orthogonal, noted as $x \Omperp y$, if $\ps{\Om y}{\Om x} = (\Om x)^T (\Om y) = 0$.
    \item A vector $x \in \Rn{n}$ is $\Om$-orthogonal to the subpsace $\K$, noted as $x \Omperp \K$, if $x \Omperp y, \; \forall y \in \K$, or equivalently $(\Om \V)^T (\Om x) = 0$ for any basis $\V$ of $\K$.
    \item A set of vectors $(v_1,\dots,v_k)$, or equivalently the matrix $\V$ formed by them, is $\Om$-orthonormal if $(\Om \V)^T (\Om \V) = I_k$.
    \item Given a subspace $\K$, we define $\POm$ as the $\Om$-orthogonal projector onto $\K$ the operator such that for $z \in \Rn{n}$ $(z - \POm z) \Omperp y, \; \forall \; y \in \K$. equivalently, $\POm z = \arg \min_{u \in \K} \norm*{\Om (z-u)}$.
\end{itemize}
We refer to $\Om$-orthonormal also as sketch-orthonormal.  The following result from \cite{BalabanovRandomizedGramSchmidt2022}[Corollary 2.2] reflects the fact that an $\epsilon$-embedding preserves geometry.
\begin{corollary}
    \label{cor:rando_sval}
    If $\Om \in \Rnm{\dsk}{n}$ is an $\epsilon$-embedding for $\Span{V}$ where $\V \in \Rnm{n}{\sa}$, then the singular values of $V$ are bounded by:
    \begin{equation}
        (1+\epsilon)^{-1/2} \sigma_{min}(\Om \V) \leq \sigma_{min}(\V) \leq \sigma_{max}(\V) \leq  \sigma_{max}(\Om \V) (1-\epsilon)^{-1/2}. 
    \end{equation} 
\end{corollary}

\section{Randomized orthogonalization processes and their usage in Arnoldi}
\label{sec:orthog}

This section introduces sketched orthogonalization processes. They lie at the heart of the randomized Arnoldi method and their numerical stability is critical for the accuracy of our randomized eigensolver. We discuss the computation of a sketch orthonormal set of vectors, the condition number of the resulting basis, the integration of the randomized Gram-Schmidt algorithm in an Arnoldi procedure, and several implementation details.

\subsection{Computing a sketch-orthonormal basis}

Given $W \in \Rnm{n}{\sa}$ and an $\epsilon$-embedding $\Om \in \Rnm{n}{\dsk}$ for $\Span{W}$, we now focus on constructing an $\Om$-orthonormal matrix $Q \in \Rnm{n}{\sa}$ such that $\Span{Q} = \Span{W}$ and an upper triangular matrix $R \in \Rnm{\sa}{\sa}$ so that $W = QR$. The method should also outputs the sketched basis $S \coloneqq \Om Q$. A generic algorithm to do so is given in Algorithm~\ref{alg:sketch-orthonormal}.
\begin{algorithm}[!htbp]
    \caption{Computation of a sketch-orthonormal basis}
    \label{alg:sketch-orthonormal}
    \begin{algorithmic}[1]
        \REQUIRE $W\in \Rnm{n}{\sa}, \Om \in \Rnm{n}{\dsk}$ the sketching matrix. 
        \ENSURE $Q \in \Rnm{n}{\sa},R \in \Rnm{\sa}{\sa}$ and $S \in\Rnm{\dsk}{\sa}$ s.t. $W = QR$ where $Q$ is $\Om$-orthonormal, $R$ is upper triangular and $S = \Om Q$.  
        \FOR{$j = 1,\dots,\sa$}
        \STATE Initialize $w_j = W(:,j)$
        \STATE Sketch $z_j = \Om w_j$
        \STATE Solve with a given method $R(1:j-1,j) = \arg \min_{y \in \Rn{j-1}} \norm*{S(:,1:j-1) y - z_j}$ \label{alg:sketch-ortho-LS}
        \STATE Compute $q_j = w_j - Q(:,1:j-1)R(1:j,j-1)$ \label{alg:qjupdate}
        \STATE Sketch $s_j = \Om q_j$ OR Compute $s_j = z_j - S(:,1:j-1)R(1:j,j-1)$ \label{alg:sketch-ortho-sjupd}
        \STATE Optional: Do re-orthogonalization by doing lines~\ref{alg:sketch-ortho-LS} to~\ref{alg:sketch-ortho-sjupd} using $q_j$, $s_j$ instead of $w_j$, $z_j$ respectively.
        \STATE Normalize $R(j,j) = \norm{s_j}$
        \STATE Store $Q(:,j) = q_j / \norm{s_j} $ and $S(:,j) = s_j / \norm{s_j}$
        \ENDFOR
    \end{algorithmic}
\end{algorithm}
We consider here different cases that are mathematically equivalent.

In RGS, the randomized version of Gram-Schmidt introduced in  \cite[Section 2.4]{BalabanovRandomizedGramSchmidt2022}, the least squares problem in line \ref{alg:sketch-ortho-LS} can be solved using a standard QR factorization of $S$ through Householder reflections for instance. By using the Matlab backslash notation, this means
\begin{equation}
   R(1:j-1,j) = S(:,1:j-1) \; \backslash \; z_j. 
\end{equation}
RGS cost is halved compared to classical Gram-Schmidt, its communication cost is similar to classical Gram-Schmidt, while the numerical stability is expected to be similar to modified Gram Schmidt, as discussed in \cite{BalabanovRandomizedGramSchmidt2022}. RGS computes $s_j$ by sketching each newly computed vector $q_j$ in line~\ref{alg:sketch-ortho-sjupd}. This  has a cost that depends on the $\epsilon$-embedding considered, for instance $O(n \log \sa)$ operations in total for SRHT or $O(n \zeta)$ for Sparse Sign, and it is more stable than its variant, referred to as vector oriented RCholeskyQR in \cite{Balabanov2021RandomizedblockGram}. This latter is equivalent to compute $R$ directly from a QR factorization of the sketch of the vectors and uses its inverse to compute the basis $Q$. It corresponds here to compute $s_j = z_j - S(:,1:j-1)R(1:j,j-1)$ on line~\ref{alg:sketch-ortho-sjupd}.

The least square problem in line \ref{alg:sketch-ortho-LS} can be solved by considering that S is orthonormal, thus 
\begin{equation}
   R(1:j-1,j) = S(:,1:j-1)^T z_j. 
\end{equation}
We refer to this solution as rCGS for randomized Classical Gram-Schmidt, given the resemblance with Classical Gram-Schmidt. It is less expensive than RGS given that it uses only the transpose of $S$, but it is often unstable in our experiments. Re-orthogonalization can be done by re-doing lines \ref{alg:sketch-ortho-LS} to \ref{alg:sketch-ortho-sjupd}. We refer to this method as rCGS2 to account for the re-orthogonalization. The second orthogonalization step adds three new computations: the matrix-vector product in line~\ref{alg:sketch-ortho-LS}, the update of $q_j$ and its sketch. To decrease the costs of the update and of the sketch, it is possible to do the first orthogonalization step by using vector oriented RCholeskyQR and not computing $q_j$. An important experimental finding is that rCGS2 outputs a better conditioned $S$ at the end than RGS, notably $\norm{I - S^T S}$ is significantly closer to $0$ for rCGS2, see the experiment in Figure~\ref{fig:orthoQR}. 

In summary, rCGS is the fastest but is unstable, RGS is the best compromise between stability and cost, and rCGS2 is slightly more expensive than RGS but produces a better conditioned sketch of the basis. The latter is especially relevant compared to RGS when the updates of $q_j$ are affordable compared to the backslash solve,  which could happen when seeking many eigenvalues for a not very large input matrix $A$. RGS is especially relevant when the least squares solver through the QR factorization is efficiently implemented as for example updating progressively the QR factorization of $S$.
Numerical results are presented in Figure~\ref{fig:orthoQR} where a numerically singular matrix from \cite{BalabanovRandomizedGramSchmidt2022}[Section 5.2] is sketched orthonormalized. It consists of $W \in \Rnm{n}{\sa}$ s.t. 
\begin{equation}
  w_{i,j} = \frac{\sin (10(\mu_j + x_i))}{\cos (100(\mu_j - x-i)) +1.1}  
\end{equation}
where $x_i$ and $\mu_j$ varies from $0$ to $1$ with equally distanced points, for $n = 10^5$ and $\sa = 300$. Here rCGS2 is the implementation with two computations of $q_j$ and two sketches of it, while rCGS2w stands for rCGS2 weak and compute $q_j$ only once as mentioned above. Note that rCGS2w, which skips line \ref{alg:qjupdate}, is unstable in this situation. The condition number of $Q$ over the iterations highlights the stability of RGS and rCGS2, while the norm $\norm*{I - S^TS}$ is only maintained close to $0$ for rCGS2. This is an important property of rCGS2, yet RGS outputs $Q$ and $S$ factors conditioned well enough for many applications.  Lastly, our experiments have shown that a vector version of the Richardson method described in \cite{Balabanov2021RandomizedblockGram} is expensive and thus we do not consider it further in this paper.
\begin{figure}
    \begin{subfigure}[b]{0.48\textwidth}
        \centering
        \includegraphics[width =\textwidth]{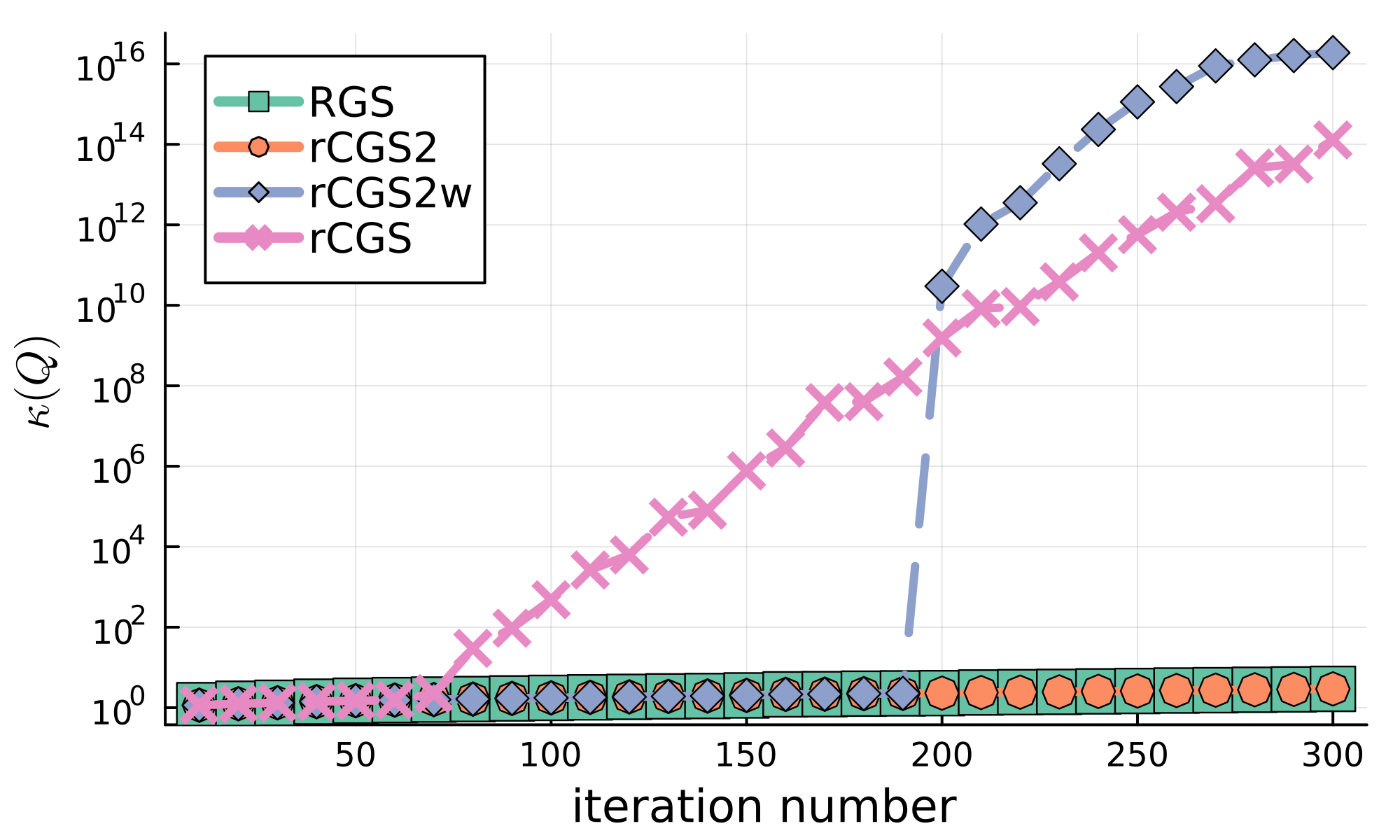}
        \caption{Condition number of $Q \in \Rnm{n}{(\# \text{ of iteration})}$}
        \label{fig:CondsQ}
    \end{subfigure}
    \hfill
    \begin{subfigure}[b]{0.48\textwidth}
        \centering
        \includegraphics[width =\textwidth]{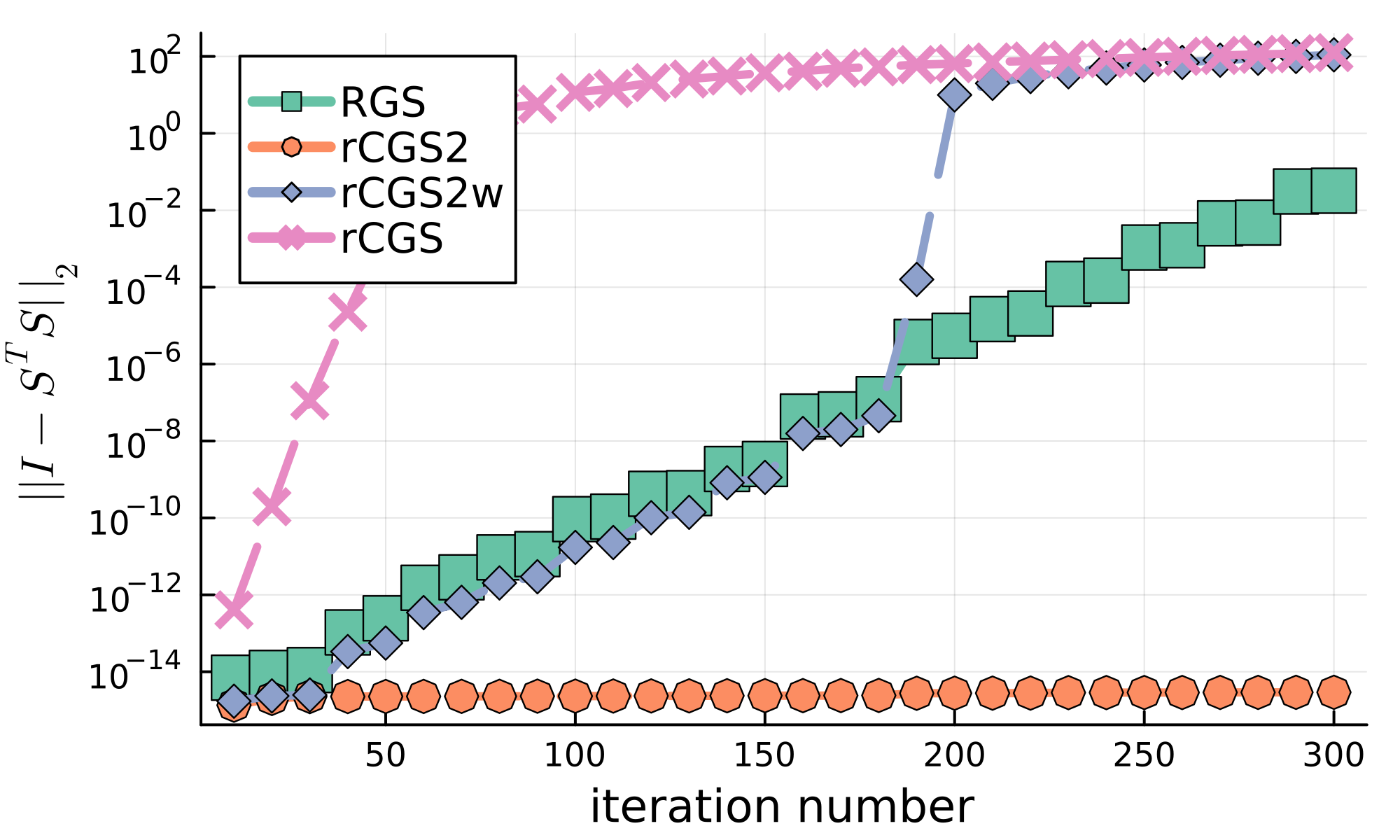}
        \caption{Operator norm $\norm{I - S^T S}$ over the iterations}
        \label{fig:I-StS}
    \end{subfigure}
    \hfill
    \caption{Sketched orthonormalization of numerically singular $W \in \Rnm{10^5}{300}$}
    \label{fig:orthoQR}
\end{figure}

\subsection{The randomized Rayleigh-Ritz procedure}

We present here a randomized Rayleigh-Ritz procedure that combines the Arnoldi procedure and sketch orthonormalization, derived notably in \cite{Balabanov2021RandomizedblockGram}[section 4] or in \cite{NakatsukasaFastAccurate2021}[section 1]. We start by defining the \rArno{} factorization.

\begin{definition}[\rArno{} factorization]
    \label{def:rArno}
    Given $\vinit \in \Rn{n}$, let the $\sa$ dimensional Krylov subspace be $\K = \Span{\vinit, \A \vinit, \dots, \A^{\sa-1}\vinit} \subset \Rn{n}$. Let $\Om$ be an $\epsilon$-embedding for $\Span{\K,\A \K}$. The computation of a sketch-orthonormal basis for $\K$ using Algorithm~\ref{alg:RCGS2-Arnoldi} for instance results in a \rArno{} factorization  
    \begin{equation}
    \label{eq:r-arnoldiRelation}
    \A \V = \V \Hsa + \rsa e_k^T, 
    \end{equation}
    with $\V \in \Rnm{n}{k}$ an $\Om$-orthonormal basis for $\K$ satisfying $\V(:,1) \coloneqq \vinit$, $\Hsa \in \Cnm{k}{k}$ an upper Hessenberg with positive sub-diagonal elements, and $\rsa = h_{\sa+1,\sa} v_{\sa +1} \in \mathcal{K}_{\sa+1}$ is $\Om$-orthogonal to $\Span{\V} = \K$.
\end{definition}
Note that in the above definition, $\K$ and $\A \K$ overlap so the subspace $\Span{\K,\A \K} = \Span{\K, v_{\sa +1}}$ is of dimension $\sa+1$. It is a conventional notation since embedding $\A \K$ will be important to derive error bounds on $\norm*{\Om (\A \rivec - \rival \rivec)}$ for $\rivec \in \K$. For the rest of the paper, when considering a \rArno{} factorization of size $\sa$, it is assumed that $\Om$ is such an $\epsilon$-embedding, unless stated otherwise when a more general OSE is required with high probability. 
In exact arithmetic, $\sigma_{min}(\Om \V) = \sigma_{max}(\Om \V) = 1$, and thus, according to Corollary~\ref{cor:rando_sval} from \cite{BalabanovRandomizedGramSchmidt2022}[Corollary 2.2], for $\epsilon = \frac{1}{2}$, $\V$ has a condition number 
\begin{equation*}
 \kappa(\V) =\frac{\sigma_{max}(\V)}{\sigma_{min}(\V)} \leq \frac{\sqrt{1+\epsilon}}{\sqrt{1-\epsilon}} \leq \sqrt{3}.   
\end{equation*}
\begin{algorithm}[!htbp]
    \caption{rCGS2-Arnoldi procedure}
    \label{alg:RCGS2-Arnoldi}
    \begin{algorithmic}[1]
        \REQUIRE $\A \in \Rnm{n}{n}$, $\sa$ the Krylov subspace dimension, $\Om \in \Rnm{\dsk}{n}$ an $\epsilon$-embedding for $\K$, $\vinit \in \Rn{n}$ with $\norm*{\Omega \vinit} =1$.
        \ENSURE $\{\vinit,\dots,v_{\sa+1}\}$ as $(\V,v_{\sa+1})$, $(\VOm,s_{\sa+1})$ the sketch of $(\V,v_{\sa+1})$, $\Hsa$ and $\rsa \coloneqq h_{\sa+1,\sa} v_{\sa+1}$   
        \STATE Initialize $V(:,1) = \vinit$, $\VOm(:,1) = \Om \vinit$
        \FOR{$j = 1,\dots,\sa$}
        \STATE Compute $w_j = A V(:,j)$
        \STATE Sketch $z_j = \Om w_j$
        \STATE Solve $H(1:j,j) = \VOm(:,1:j)^T z_j$ 
        \STATE Compute $v_j = w_j - \V(:,1:j)H(1:j,j)$
        \STATE Sketch $s_j = \Om v_j$ 
        \STATE // Re-orthogonalization 
        \STATE $\tilde{H_j} = \VOm(:,1:j)^T s_j$ 
        \STATE Update $v_j = v_j - \V(:,1:j)\tilde{H_j}$
        \STATE Update $s_j = s_j - \VOm(:,1:j)\tilde{H_j} $ 
        \STATE Update $H(1:j,j) = H(1:j,j) + \tilde{H_j}$
        \STATE Normalize $H(j+1,j) = \norm{s_j}$, stop if $\norm{s_j} =0$
        \STATE Store $\V(:,j+1) = v_j / \norm{s_j} $ and $\VOm(:,j+1) = s_j / \norm{s_j}$
        \ENDFOR
    \end{algorithmic}
\end{algorithm}
Algorithm~\ref{alg:RCGS2-Arnoldi} presents a randomized Arnoldi process using rCGS2 for the sketch-orthogonalization step. It outputs a basis $\V$ for $\K$, as it can be seen by induction, since every $v_j$ is equal to a polynomial of $\A$ of degree $j-1$ applied to $\vinit$, and the vectors are linearly independent as $\kappa(\V) \leq \sqrt{3}$. It also outputs the sketch $\VOm$ of this basis that is stored such that the sketches can be reused. The happy breakdown condition remains valid since
$$\norm*{s_j} = 0 \implies (1-\epsilon) \norm*{v_j}^2 \leq 0 \leq (1+\epsilon) \norm*{v_j}^2 \implies \norm*{v_j} = 0 \implies r_j = 0$$
by the $\epsilon$-embedding property \eqref{eq:epsembedd} and so $\V \in \Rnm{n}{j}$ is an invariant subspace of $\A$.
Note that all the variants of Algorithm~\ref{alg:sketch-orthonormal} for computing a sketch-orthogonal basis are mathematically equivalent.
Hence, we use the term \rArno{} to refer to this process, regardless of the specific algorithm used for the computation of the sketch-orthogonal basis. This process can be used to compute eigenpairs of $\A$ using the randomized Rayleigh-Ritz procedure defined below.
\begin{definition}[randomized Rayleigh-Ritz using \rArno{}]
    \label{def:sRR}
    Let $\A \V = \V \Hsa + \rsa e_k^T$ be a \rArno{} factorization. 
    From eigenpairs $(\yrivec_i, \rival_i)$ of $\Hsa \in \Rnm{k}{k}$, the Ritz pairs of $\A$ are the pairs $(\rivec_i \coloneqq \V \yrivec_i, \rival_i)$ satisfying the sketched Petrov-Galerkin condition,
    \begin{equation}
        \A \rivec_i - \rival_i \rivec_i \Omperp \K \; 
     \text{with} \; \rivec_i \in \K, \; \norm{\Omega \rivec_i} = 1, \; \rival_i \in \C.
     \label{eq:rando-rrEVP}
    \end{equation}
\end{definition}
We define an oblique projector $\POm$ on $\K$ as $\POm x = \arg \min_{y \in \K} \norm*{\Om(x-y)}$ for $x \in \Rn{n}$.
It is shown in \cite{Balabanov2021RandomizedblockGram} that the condition \eqref{eq:rando-rrEVP} of the randomized Rayleigh-Ritz process implies that the Ritz pairs are the eigenpairs of the following eigenvalue problem obtained through an oblique projection,
\begin{equation}
   \POm \A \POm \rivec_i = \rival_i \rivec_i.
\end{equation}
Moreover, it is shown in \cite{NakatsukasaFastAccurate2021} that randomized Rayleigh-Ritz can also be derived from the minimization
 \begin{equation}
     \V \Hsa = \arg \min_{Y, \Span{Y} \subset \K} \norm{\Om(\A \V - Y)} = \POm \A \V.
 \end{equation}
These two results are randomized analogs of the deterministic Rayleigh-Ritz procedure from \cite{saad2011numerical}[Chapter 4] and do not require $\K$ to be a Krylov subspace. In this article, we focus on randomized Rayleigh-Ritz using the \rArno{} method.

\section{Randomized Implicitly Restarted Arnoldi}
\label{sec:rIRA}

This section introduces the randomized implicitly restarted Arnoldi method.  It starts by motivating the use of \rArno{} within an implicitly restarted procedure, and then it introduces the resulting novel algorithm. We discuss in details its different steps and we notably show that the \rArno{} relation is maintained during the algorithm. Then we discuss the stopping criterion of the algorithm based on the sketched residual error issued from the \rArno{} process and show its relation with the standard Arnoldi residual error. 

\subsection{Main algorithm}
We show first that, as in the deterministic case, the \rArno{} procedure produces an invariant subspace if the starting vector lies in the span of $\Om$-orthonormal Schur vectors of $\A$ (defined below).  Based on this result, the underlying idea of randomized Implicitly Restarted Arnoldi method is to drive the starting vector $\vinit$ towards this span by updating the \rArno{} factorization using the deterministic shifted QR algorithm performed on the associated Hessenberg matrix. We first define the $\Om$-orthonormal Schur factorization of $\A$ and associated $\Om$-orthonormal Schur vectors.
\begin{definition}
    A $\sa$ partial $\Om$-orthonormal Schur factorization of $\A$ is the decomposition 
    \begin{equation}
      \A Q = Q R,  
    \end{equation}
    where $Q \in \Rnm{n}{\sa}$ is $\Om$-orthonormal and $R$ is block upper-triangular, with blocks of size up to $2 \times 2$. Columns of $Q$ are called $\Om$-orthonormal Schur vectors of $\A$ by analogy with the Schur vectors, and are noted $q_\sa$. It holds that
    \begin{equation}
        \A q_\sa = \eival_\sa q_\sa + \sum_{i=1}^{\sa-1} r_{i,k} q_i.
    \end{equation}
\end{definition}
It can be easily shown that the eigenvalues of $R$ are a subset of the eigenvalues of $\A$. Indeed, given the eigenpair $(y,\lambda)$ for $R$, one has: 
\begin{equation}
    R y = \lambda y \implies A \frac{(Q y)}{\norm{Qy}} = \lambda \frac{(Q y)}{\norm{Qy}},
\end{equation}
such that $(Q y / \norm{Qy}, \lambda) $ is an eigenpair for $\A$.
\begin{theorem}
    Let $\A \V = \V \Hsa +   \rsa e_\sa^T$ be a $\sa$-step \rArno{} factorization as in Definition~\ref{def:rArno}, with $\Hsa$ unreduced. Then $\rsa = 0$ and $\V$ spans an invariant subspace for $\A$ if and only if $V(:,1) = \vinit = Qy$, where $\A Q = QR$ is a $\sa$ partial $\Om$-orthonormal Schur factorization of $\A$ with $Q \in \Rnm{n}{\sa}$ and for some $y \in \Rn{\sa}$.
    \label{th:schur-invariant}
\end{theorem}
\begin{proof}
    First, let us assume that $\rsa = 0$. We denote by $\Hsa Q_k = Q_k R_k$ the Schur factorization of $\Hsa$, with $Q_k \in \Rnm{k}{k}$ square orthonormal. Then
    $ \A \V Q_k = \V \Hsa Q_k = \V Q_k R_k$.
    Define $Q \coloneqq \V Q_k$. Then $Q$ is $\Om$-orthonormal, since we have:
    $$ (\Om Q)^T (\Om Q) = (\Om \V Q_k)^T (\Om \V Q_k)  = Q_k^T (\Om \V)^T (\Om \V) Q_k = Q_k^T Q_k = I_k. $$
    We thus obtain $\A Q = Q R_k$, a partial $\Om$-orthonormal Schur factorization of $\A$, and $\vinit = \V e_1 = \V Q_k Q_k^T e_1 = \V Q_k y = Qy$ with $y = Q_k^T e_1$.  

    Now assume $\vinit = Qy$ with $\A Q = Q R$. We show that $\dim(\mathcal{K}_{\sa+1}) = \sa$. We have $\A^m Q = Q R^m$ for every integer $m \geq 0$, and so $\A^m \vinit = \A^m Q y = Q R^m y \in \Span{Q}, \; \forall m \geq 0$. Notably, this means that $\mathcal{K}_{\sa+1}(\A, \vinit) \coloneqq \Span{\vinit,\dots,\A^{\sa}\vinit} \subset \Span{Q}$, thus $\dim(\mathcal{K}_{\sa+1}) \leq \dim(\Span{Q}) = \sa$, the last equality coming from $Q \in \Rnm{n}{k}$ being $\Om$-orthonormal and thus full rank thanks to Corollary~\ref{cor:rando_sval}. 
    In addition, Krylov subspaces satisfy the nesting property $\K \subset \mathcal{K}_{\sa+1}$. This implies that $\dim(\mathcal{K}_{\sa+1}) \geq \dim(\K)$. We have that $\dim(\K) = \sa$ since $\Hsa$ is unreduced, meaning $\K$ is not degenerate. Thus $\dim(\mathcal{K}_{\sa+1}) = k$, and so we have $\K = \mathcal{K}_{\sa+1}$ by inclusion and dimension equality. To finally show that $\rsa = 0$, write the decomposition $\rsa = \sum_{j=1}^{\sa} \alpha_j v_j$ since the vectors $v_j$ for $j=1,\dots,\sa$ are a set of $\sa$ linearly independent vectors of $\mathcal{K}_{\sa+1}$ and thus they form a basis for $\mathcal{K}_{\sa+1}$. The $\alpha_i$ can be obtained as $\ps{\Om \rsa}{\Om v_i} = \alpha_i$ using the $\Om$-orthonormality of the basis $\V$, and since $\rsa \Omperp \K$, we have $\ps{\Om \rsa}{\Om v_i} = 0$ for all $i = 1,\dots,\sa$. This gives $\rsa = 0$.
\end{proof}
We will see later in the analysis section~\ref{sec:analysis} the relation between  $\rsa$ and $\vinit$ in \rArno{}, see Theorem~\ref{th:rsa-vinit}, and how $\vinit$ is updated to be in the span of approximate $\Om$-orthonormal Schur vectors of $\A$, see Equation~\eqref{eq:approxOmSchur}. 
The randomized Implicitly Restarted Arnoldi process is presented in Algorithm~\ref{alg:rIRA}. The main difference with Algorithm~\ref{alg:IRA} lies in the use of \rArno{} to obtain the Arnoldi factorization. We describe now in more details its main steps. 

\begin{algorithm}[!htbp]
    \caption{randomized Implicitly Restarted Arnoldi (rIRA)}
    \label{alg:rIRA}
    \begin{algorithmic}[1]
        \REQUIRE $\A \in \Rnm{n}{n}$, $\ba$ the Krylov dimension, $\Om \in \Rnm{\dsk}{n}$ an  $\epsilon$-embedding for $ \Span{\mathcal{K}_{\ba}, \A \mathcal{K}_{\ba}} $, $\vinit \in \Rn{n}$ with $\norm*{\Omega \vinit} =1$ and a convergence criteria. 
        \ENSURE $(\vinit,\dots,v_{\ba+1})$ as $(\V,v_{\ba+1})$, $\Hba$ and $\rba \coloneqq h_{\ba+1,\ba} v_{\ba+1}$
        \STATE Perform $\sa$ steps of the \rArno{} procedure to obtain 
        $\A \V = \V \Hsa + \rsa e_{\sa}^T$ and $\VOm = (\Om \V) \in \Rnm{\dsk}{\sa}$ \label{alg:rIRA-arnoldi-Init}
        \WHILE{convergence not declared} 
        \STATE Extend the \rArno{} factorization through $\nshi$ additional steps: $\A \V = \V \Hba + \rba e_{\ba}^T$ and $\VOm = (\Om \V) \in \Rnm{\dsk}{(\ba)}$ \label{alg:rIRA-arnoldi-extend}
        \STATE Compute the eigenvalues $(\rival_1, \dots, \rival_{\ba})$ of $\Hba$. Either declare convergence using some of these and stop, or select $\nshi$ shifts $(\shi_1,\dots,\shi_\nshi)$ among them.  
        \label{alg:rIRA-get-shifts}
        \STATE // Initialize $\Hbap$ to $\Hba$ and perform $\nshi$ shifted QR steps on it:
        \FOR{$i = 1,\dots, \nshi$} \label{alg:rIRA-qr-start}
        \STATE $(\Hbap - \shi_i I) = Q^i R^{i}$ \label{alg:rIRA-qrfacto}
        \STATE $\Hbap = (Q^{i})^T \Hbap Q^i$ \hspace{0.8cm}// Note that $(Q^{i})^T \Hbap Q^i= R^{i} Q^i + \shi_i I$ which is less expensive
        \ENDFOR \label{alg:rIRA-qr-end}
        \STATE Set $\Q = Q^1 \dots Q^{\nshi}$ and $\Vplus = \V \Q$ \label{alg:rIRA-update}
        \STATE Set $\rsaplus = \Hbap(\sa+1,\sa) \cdot \Vplus(:,k+1) + \Q(\ba,\sa) \cdot \rba$ \label{alg:residualupdate}
        \STATE Truncate $\Hsaplus = \Hbap(1:\sa,1:\sa)$ and $\Vplus = \Vplus(:,1:\sa)$ \label{alg:rIRA-truncate}
        \STATE // Continue with the resulting \rArno{} factorization $\A \Vplus = \Vplus \Hsaplus + \rsaplus e_{\sa}^T$ and  $\Tilde{\VOm} = \Om \Vplus \in \Rnm{\dsk}{\sa}.$
        \ENDWHILE 
    \end{algorithmic}
\end{algorithm}

Lines~\ref{alg:rIRA-arnoldi-Init} and~\ref{alg:rIRA-arnoldi-extend} correspond to the randomized Arnoldi factorizations. Obtaining such factorizations is discussed in section~\ref{sec:orthog}. The extension of a given Arnoldi factorization from $\sa$ to $\ba$ is well explained in \cite{Lehoucq1995DeflationTechniquesImplicitly} or \cite{Sorensen1992ImplicitApplicationPolynomial}. It requires to augment $\V$ and $\Hsa$ with $v_{\sa+1}$ and $h_{\sa+1,\sa}$ respectively and continue an Arnoldi procedure, and is easily adaptable to the randomized case. 

Line~\ref{alg:rIRA-get-shifts} corresponds to the shifts selection, which lies at the heart of the shifted QR algorithm. This step defines the polynomial that is implicitly applied to $\A$ when restarting, as described in \cite{Sorensen1992ImplicitApplicationPolynomial} or in \cite{saad2011numerical}[Chapter 7]. It is indeed shown that IRA updates $\vinit$ with 
$$\vinitj = \polp(\A)\vinit, \quad \text{where} \quad \polp(\lambda) \coloneqq \prod_{i= 1}^{\nshi}(\lambda - \shi_i),$$
such that $\polp$ is a polynomial of degree $\nshi$ that has the shifts as roots. If one considers the shifts as approximations of eigenvalues of $\A$, it has the consequence to severely reduce components of $\vinit$ along eigenvectors associated with $\eival_i \approx \shi_i, i = 1 ,\dots,\nshi$.
The strategy referred to as exact shifts selection consists in dividing the spectrum of $\Hba$ into two subsets, wanted Ritz eigenvalues and unwanted Ritz eigenvalues. The shifts correspond to the unwanted ones. This allows the user to specify which part of the spectrum of $\A$ is of interest: smallest (or largest) modulus eigenvalues, or their real or imaginary part for instance. This important property of IRA is preserved in rIRA, as discussed later in the analysis, see Equation~\eqref{eq:vinitj-polpA}.

Lines~\ref{alg:rIRA-qr-start} to~\ref{alg:rIRA-qr-end} correspond to the $\nshi$ shifted QR steps. Note that there exist strategies that allow to maintain real arithmetic even when dealing with complex shifts, as explained in \cite[Francis QR Step]{Golub2013Matrixcomputations}.  We emphasize that the $\nshi$ QR factorizations should be computed by using a deterministic QR. Since $\ba$ is small compared to $n$, deterministic QR is affordable and it can be computed using Householder QR for instance, see e.g. \cite{Golub2013Matrixcomputations}.  Randomized orthogonalization processes are not suitable in this case since they require the sketch dimension $\dsk$ to be larger than the number of vectors when the $\epsilon$-embedding property should be satisfied for a priori unknown subspaces. 

Lines~\ref{alg:rIRA-update} to \ref{alg:rIRA-truncate} correspond to the update of the Arnoldi factorization. The computation of the update is obtained from the following derivation. We start from the right multiplication of the length $\ba$ \rArno{} relation by $\Q$,
\begin{align}
    \label{eq:extended-arnoldi}
    (\A \V) \Q & = (\V \Hba) \Q + (\rba e_{\ba}^T ) \Q, \nonumber \\
    \A \V \Q & = \V \Q \Hbap + \rba e_{\ba}^T \Q, 
\end{align}
since $ \Q^T \Hba \Q = \Hbap$. We recall that a QR decomposition of an upper Hessenberg matrix outputs a $Q$ factor also upper ? Hessenberg. This is because for $H$ upper Hessenberg, $H = QR \implies Q = H R^{-1}, $
and the product of an upper triangular matrix $R^{-1}$ with an upper Hessenberg $H$ outputs an upper Hessenberg factor $Q$. Thus each $Q^i \in \Rnm{(\ba)}{(\ba)}$ is upper Hessenberg, and so $\nshi$ products of them produce a matrix whose last row is zero for the first $(\ba) - \nshi - 1 = \sa -1$ components, that is $e_{\ba}^T \Q = (0 \dots 0 \; \eta_{\sa} \dots \eta_{\ba})$. At the same time, we also have the important property that $\Hbap$ is still upper Hessenberg, given that the product $R^i Q^i$ is, with $R^i$ upper triangular. We can thus partition $\Hbap =
\begin{pmatrix}
    \Hsaplus & M \\
    \beta e_1 e_{\sa}^T & H_p 
\end{pmatrix},$ with $\beta \coloneqq \Hbap(\sa+1,\sa)$ and $\Hsaplus$ upper Hessenberg. 
Finally, by equating the first $\sa$ columns of Equation~\eqref{eq:extended-arnoldi} we obtain:
\begin{equation}
    \A \V \Q(:,1:\sa) = \V \Q(:,1:\sa) \Hsaplus +  \V \Q(:,\sa+1) \beta e_{\sa}^T + \rba \eta_{\sa} e_{\sa}^T.
    \label{eq:updated-arnoldi}
\end{equation} 
To show that this is a legitimate \rArno{} relation, we verify the $\Om$-orthonormality of the new Krylov basis $\V \Q(:,1:\sa)$ and $\Om$-orthogonality with respect to it of the new residual: 
$$(\Om \V \Q(:,1:\sa))^T (\Om \V \Q(:,1:\sa)) = \Q(:,1:\sa)^T (\Om \V)^T (\Om \V) \Q(:,1:\sa) = \Q(:,1:\sa)^T \Q(:,1:\sa) = I_{\sa},$$ 
and
\begin{align*}
    (\Om \V \Q(:,1:\sa))^T (\Om (\beta \V \Q(:,\sa+1) + \rba \eta_{\sa})) & = \beta \Q(:,1:\sa)^T (\Om \V)^T (\Om \V) \Q(:,\sa+1) \\
    & +  \eta_{\sa} \Q(:,1:\sa)^T (\Om \V)^T (\Om \rba) \\
    & = \beta \Q(:,1:\sa)^T \Q(:,\sa+1) \\
    & = 0,
\end{align*}
using $(\Om \V)^T (\Om \rba) = 0$ by $\Om$-orthogonality of $\rba$ to $\V$, and $\Q(:,1:\sa)^T \Q(:,\sa+1) = 0$ by orthogonality of $\Q$. Equation~\eqref{eq:updated-arnoldi} can then be written as
\begin{equation}
    \A \Vplus = \Vplus \Hsaplus + \rsaplus e_{\sa}^T,
\end{equation}
with $\Vplus \coloneqq \V \Q(:,1:\sa) \in \Rnm{n}{\sa}$ $\Om$-orthonormal and $\rsaplus \coloneqq \Hbap(\sa+1,\sa) \cdot \V \Q(:,\sa+1) + \Q(\ba,\sa) \cdot \rba $. This \rArno{} factorization can thus be used for the next iteration. 

\subsection{Monitoring the residual error}
The stopping criterion of the algorithm can be defined
as reaching an accuracy $\eta$ for all the residual errors among the wanted pairs, that is when all Ritz pairs satisfy 
\begin{equation}
\label{eq:stopping-criterion}
\frac{\norm*{\A \rivec - \rival \rivec}}{\norm*{\rivec}} \leq \eta,
\end{equation}
where the division by $\norm{\rivec}$ allows normalize the vector. For a Ritz pair $(\rival, \rivec = \V \yrivec)$, we can multiply Equation~\eqref{eq:r-arnoldiRelation} with $\yrivec$ and obtain
\begin{equation}
   \A \rivec - \rival \rivec = h_{\sa+1,\sa} v_{\sa +1} e_{\sa}^T \yrivec 
\end{equation}
for a length $\sa$ \rArno{} factorization. We have $\norm*{\Om \rivec} = \norm{\Omega \V \yrivec} =  1$ since $(\Omega \V)$ is orthogonal and $\yrivec$ is an eigenvector of norm 1. Moreover $\norm*{\Om v_{\sa +1}} = 1$ by construction.  For now we divide by $\norm{\rivec}$ to normalize $\rivec$, which is required for an approximate eigenvector:
\begin{equation}
    \frac{\norm*{\A \rivec - \rival \rivec}}{\norm*{\rivec}} = h_{\sa+1,\sa} \frac{\norm*{v_{\sa +1}}}{\norm*{\rivec}} \abs{e_{\sa}^T \yrivec}.
    \label{eq:r-residualNorm}
\end{equation}
It is the analogous of Equation~\eqref{eq:residualNorm} for the residual error. The right hand side of Equation~\eqref{eq:r-residualNorm} is computable and cheaper to get than its left hand side counterpart, but it adds the computation of two norms in $\Rn{n}$ compared to the deterministic case, namely $\norm*{v_{\sa +1}}$ and $\norm*{\rivec}$. To avoid these computations, bounds can be derived in the randomized setting as shown in \cite{NakatsukasaFastAccurate2021}. With $\Om$ an $\epsilon$-embedding for $\Span{\K,\A \K}$, we use the $\epsilon$-embedding property from Equation~\eqref{eq:epsembedd} to obtain 
\begin{alignat*}{2}
    \frac{1}{1 + \epsilon}\norm*{\Om(\A \rivec - \rival \rivec)}^2 & \leq   \norm*{\A \rivec - \rival \rivec}^2  && \leq  \frac{1}{1 - \epsilon}\norm*{\Om(\A \rivec - \rival \rivec)}^2,  \\
    1-\epsilon = \frac{1-\epsilon}{\norm{\Om \rivec}^2}    & \leq \hfill \frac{1}{\norm*{\rivec}^2} \hfill && \leq  \frac{1+\epsilon}{\norm{\Om \rivec}^2} = 1+\epsilon, 
\end{alignat*}
which leads to
\begin{equation}
    \sqrt{\frac{1-\epsilon}{1 + \epsilon}} \norm*{\Om(\A \rivec - \rival \rivec)} \leq  \frac{\norm*{\A \rivec - \rival \rivec}}{\norm*{\rivec}}   \leq  \sqrt{\frac{1+\epsilon}{1 - \epsilon}} \norm*{\Om(\A \rivec - \rival \rivec)}.
    \label{eq:r-residualNorm-free}
\end{equation}
It is important to note that $\Om(\A \rivec - \rival \rivec) = \Om h_{\sa+1,\sa} v_{\sa +1} e_{\sa}^T \yrivec$ such that 
\begin{equation}
    \label{eq:cheapOmResid}
    \norm*{\Om(\A \rivec - \rival \rivec)} = h_{\sa+1,\sa} \abs{e_{\sa}^T \yrivec},
\end{equation}
i.e.\@ this is a quantity available at no significant cost. Consequently, we use in our work the stopping criterion
\begin{equation}
    h_{\sa+1,\sa} \abs{e_{\sa}^T \yrivec} \leq \Tilde{\eta}
\end{equation}
for a given target accuracy $\Tilde{\eta}$, knowing that the residual error of interest $\norm*{\A \rivec - \rival \rivec} / \norm{\rivec}$ is close to it up to factors $\sqrt{(1 \pm \epsilon)/(1 \mp \epsilon)}$. That is  
\begin{equation}
     \frac{\norm*{\A \rivec - \rival \rivec}}{\norm*{\rivec}}   \leq \sqrt{\frac{1+\epsilon}{1 - \epsilon}} \Tilde{\eta} = \eta
\end{equation}
and these small factors do not change the order of magnitude, so if $\Tilde{\eta}$ is small then $\eta$ is small and the Ritz pair can be declared as converged. 

We note that another stopping criterion can be
\begin{equation}
    \frac{\norm*{\A \rivec - \rival \rivec} / \norm*{\rivec}}{\norm{\A \rivec} / {\norm{\rivec}}} \leq \hat{\eta}.
\end{equation}
which allows to work with relative error and can be relevant when considering finite precision arithmetic. Using again the $\epsilon$-embedding property of $\Om$, we obtain
\begin{equation}
    \sqrt{\frac{1}{1 + \epsilon}} \frac{\norm*{\Om(\A \rivec - \rival \rivec)}}{\norm{\A \rivec}} \leq  \frac{\norm*{\A \rivec - \rival \rivec}}{\norm*{\A \rivec}}   \leq  \sqrt{\frac{1}{1 - \epsilon}} \frac{\norm*{\Om(\A \rivec - \rival \rivec)}}{\norm{\A \rivec}}.
    \label{eq:r-relative-residualNorm-free}
\end{equation}
We do not use this stopping criterion in our experiments, however we show that this is also easily computable using sketching.
The quantity $\norm*{\Om(\A \rivec - \rival \rivec)} = h_{\sa+1,\sa} \abs{e_{\sa}^T \yrivec}$ is readily available. In addition, if $\A \rivec$ is approximated by $ \rival \rivec$, which is a good approximation close to the convergence of the Ritz pair, then $h_{\sa+1,\sa} \abs{e_{\sa}^T \yrivec}/ |\rival| $ is inexpensive to compute.  

\subsection{Implicit sketching and sketched vectors storage}
We finally discuss the computation of the sketched quantities in the rIRA framework when updating the \rArno{} factorization and truncating it as in Lines~\ref{alg:rIRA-update} to~\ref{alg:rIRA-truncate} of Algorithm~\ref{alg:rIRA}. We already discussed for Algorithm~\ref{alg:RCGS2-Arnoldi} that the sketch of $\V$, which is noted $\VOm$, is stored during the algorithm to avoid numerous sketching operations. It is thus an output of every \rArno{} factorization. When extending the factorization from size $\sa$ to $\ba$ in Algorithm~\ref{alg:rIRA}, a $\VOm \in \Rnm{\dsk}{\sa}$ is also required as an input to carry the sketched orthogonalization steps against all $\sa$ first previous vectors. It is possible to recover $\Tilde{\VOm}$ after the update step $\Vplus \coloneqq \V \Q$. The straightforward strategy is to sketch $\Vplus$, but this requires $\ba$ computations on vectors of size $n$. However using
\begin{equation}
  \Tilde{\VOm} = \VOm \Q  \in \Rnm{\dsk}{(\ba)}
\end{equation}
before truncation, given that $\Tilde{\VOm} \coloneqq \Om \Vplus = \Om \V \Q = \VOm \Q$, is cheaper while being numerically stable given that $\Q$ is well conditioned. 
The same process can be applied for $(\Om \rsaplus)$, which is needed as the next vector for $\Tilde{\VOm}$ in the expansion step. We apply $\Om$ to the update in Line~\ref{alg:residualupdate} to obtain
\begin{equation}
   (\Om \rsaplus) =  \Hbap(\sa+1,\sa) \cdot \Tilde{\VOm}(:,k+1) + \Q(\ba,\sa) \cdot (\Om \rba). 
\end{equation}
Note that $\Tilde{\VOm}(:,k+1)$ is readily available thanks to the discussion above and $(\Om \rba)$ is given as an output of the previous \rArno{} extension run.  Finally, $\Tilde{\VOm}$ can be truncated,
\begin{equation}
  \Tilde{\VOm} = \Tilde{\VOm}(:,1:\sa) \in \Rnm{\dsk}{\sa}.  
\end{equation}
We call this implicit sketching since it is mathematically equivalent to sketching $\VOm$, but the computations are done in $\Rnm{\dsk}{(\ba)}$.

\section{Analysis of randomized Arnoldi and randomized Implicitly Restarted Arnoldi}
\label{sec:analysis}
In this section, we start by considering an oblique projection method and show an optimality result satisfied by the characteristic polynomial of the Hessenberg factor. Randomized Arnoldi is a specific case of this method. Using this oblique projection, we establish bounds on the distance between the eigenvectors of $\A$ and the Krylov subspace $\K$. We then consider specifically the rIRA method that uses \rArno{} and the deterministic shifted QR algorithm, and notably prove that rIRA does not add error to the approximation of eigenvectors when restarting.  We also specify the subspace to which $\vinit$ belongs after the update step. Finally, we give a convergence result for a specific shift selection setting.

\subsection{Analysis of \rArno{}} 
Given any $\Om \in \Rnm{\dsk}{n}$ with full row rank, we start by considering an oblique projector $\POm$ on any subspace $\K$ of dimension $\sa$, defined by 
\begin{equation}
    \POm x = \arg \min_{y \in \K} \norm*{\Om(x-y)}
\end{equation} for $x \in \Rn{n}$. It can be represented in matrix form as:
\begin{equation}
    \POm = \V (\Om \V)^\dag \Om,
\end{equation}
where $\V$ is a basis for $\K$ and $\dag$ denotes the pseudo-inverse. We derive the following property of this oblique projector.
\begin{lemma}
    \label{eq:self-adj}
    Assuming that $(\Om \V)$ is of full column rank, the oblique projector $\POm$ on $\K$ defined by $\POm x = \arg \min_{y \in \K} \norm*{\Om(x-y)}$ satisfies that 
    \begin{equation}
     \psOm{\POm x}{y} = \psOm{x}{\POm y},   
    \end{equation}
    for all  $x,y \in \Rn{n}$.
\end{lemma}
\begin{proof}
    We have that $\ps{\Om \POm x}{\Om y} = \ps{\Om \V (\Om \V)^\dag \Om x }{\Om y} = \ps{\Om x }{(\Om \V (\Om \V)^\dag)^T \Om y} = \ps{\Om x }{((\Om \V)^\dag)^T (\Om \V)^T \Om y}$. 
    Now using that $(\Om \V)$ is full column rank, the pseudo-inverse can be written as $(\Om \V)^\dag = [(\Om \V)^T (\Om \V) ]^{-1} (\Om \V)^T $. Thus
    \begin{align*}
        ((\Om \V)^\dag)^T  = (\Om \V) ([(\Om \V)^T (\Om \V) ]^{-1})^T 
         = (\Om \V) ([(\Om \V)^T (\Om \V) ]^{T})^{-1} 
         = (\Om \V) [(\Om \V)^T (\Om \V) ]^{-1}
    \end{align*}
    using commutation of the inverse and the transpose. This gives
    \begin{align*}
        ((\Om \V)^\dag)^T (\Om \V)^T \Om  = (\Om \V) [(\Om \V)^T (\Om \V) ]^{-1} (\Om \V)^T \Om 
         = \Om \V (\Om \V)^\dag \Om = \Om \POm 
    \end{align*}
    and thus $\ps{\Om \POm x}{\Om y} = \ps{\Om x }{\Om \POm y}$.
\end{proof}

We now consider the use of \rArno{}, as defined in Definition~\ref{def:rArno}, within a randomized Rayleigh-Ritz procedure, as defined in Definition~\ref{def:sRR}. This means that $\K$ is a Krylov subspace of basis $\V$ and $\Om$ is an $\epsilon$-embedding for this subspace. In this setting, note that $(\Om \V)^\dag = (\Om \V)^T$ since $\V$ is $\Om$-orthonormal. We establish the following optimality result for the \rArno{} procedure, which generalizes \cite{saad2011numerical}[Theorem 6.1] to oblique projections in the randomized context. This result will be useful to derive a convergence result for rIRA later in Theorem~\ref{th:rIRAconv}. 
\begin{theorem}
    \label{th:charac-poly-min}
    Suppose $\A \V = \V \Hsa + \rsa e_\sa^T$ is a \rArno{} factorization obtained from starting vector $\vinit$ as in Definition \ref{def:rArno}. Then $\Hsa$ is the representation in $\Rn{\sa}$ with respect to the basis $\V$ of the randomized projection $\POm \A \POm$ restricted to $\K$, that is
    \begin{equation}
        \POm \A \POm x = \V \Hsa y \quad \text{when $x = Vy$}.
    \end{equation}
    Moreover, its characteristic polynomial $\hat{p}_{\sa}$ minimizes $\norm*{\Om p(\A) \vinit}$ over the set $\mathcal{PM}_\sa$ of all monic polynomials $p$ of degree $\sa$, i.e.\@ 
    \begin{equation}
        \hat{p}_{\sa} = \arg \min_{p \in \mathcal{PM}_\sa } \norm*{\Om p(\A) \vinit}
    \end{equation}
\end{theorem}
\begin{proof}
    For the first part of the result, multiply the \rArno{} factorization by $(\Om \V)^T \Om$ and use the $\Om$-orthonormality of $\V$ and $\rsa \Omperp \V$ to obtain
    $$\Hsa = (\Om \V)^T \Om \A \V. $$
    Using the matrix form of $\POm$, one has
    $$\POm \A \POm  = \V (\Om \V)^T \Om \A \V (\Om \V)^T \Om  = \V \Hsa (\Om \V)^T \Om. $$
    Then, for any vector $x \in \K$, there exists a $y \in \Rn{k}$ such that $x = Vy$, and thus 
    $$\POm \A \POm x = \V \Hsa (\Om \V)^T \Om \V y = \V \Hsa y,$$
    hence the representation property of $\Hsa$ w.r.t. $\POm \A \POm$ and the basis $\V$.
    Now, denote $\A_\sa^\Om \coloneqq \POm \A_{|_{\K}} $ the projected restriction of $\A$ to $\K$. It coincides with $\POm \A \POm$ on $\K$, that is
    $$\POm \A_{|_{\K}} x = \POm \A \POm x \quad \forall x \in \K.$$ 
    We note $\cpol{M}{X}$ the characteristic polynomial of a matrix $M$ in the indeterminate $X$. By using the commutation property $\cpol{MN}{X} =  \cpol{NM}{X} X^{n-\sa}$ of the characteristic  polynomial between matrices $M = (\V \Hsa) \in \Rnm{n}{\sa}$ and $N = (\Om \V)^T \Om \in \Rnm{\sa}{n}$ we obtain 
    $$\cpol{\POm \A \POm}{X}= \cpol{\V \Hsa (\Om \V)^T \Om}{X}=  \cpol{(\Om \V)^T \Om \V \Hsa}{X}  X^{n-\sa} =  \cpol{\Hsa}{X} X^{n-\sa}.$$  
    Using Cayley-Hamilton theorem such that $\cpol{\POm \A \POm}{\POm \A \POm} = 0$, we  derive for any vector $v \in \K$, 
    $$\cpol{\POm \A \POm}{\POm \A \POm}v = \cpol{\POm \A \POm}{\A_\sa^\Om }v = \cpol{\Hsa}{\A_\sa^\Om}  {\A_\sa^\Om}^{n-m} v  = 0. $$
    Since ${\A_\sa^\Om}^{n-m} v \in \K$ and assuming $\A$ full rank we obtain $\cpol{\Hsa}{\A_\sa^\Om} v_i = 0$ for a basis $\{ v_i \}$ of $\K$, from which we conclude that $\cpol{\Hsa}{\A_\sa^\Om} \tilde{v} = 0$ for every $\tilde{v} \in \K$. Now let $w \in \K$ and define $\hat{p}_{\sa} \coloneqq p_{\Hsa}$. Since $\vinit \in \K$, we have
    $$ \psOm{\hat{p}_{\sa}(\A_\sa^\Om) \vinit}{w} = 0.$$
    We now use \cite[Proposition 6.4]{saad2011numerical} which states that for any projector $P$ onto $\K$ and any polynomial $q$ of degree $\deg(q) \leq \sa$, we have $q(P \A_{|_{\K}}) \vinit = P q(\A) \vinit $.
    With $q =\hat{p}_{\sa}$ and $P = \POm$, we obtain
    $$ \psOm{\POm \hat{p}_{\sa}(\A) \vinit}{w} = 0,$$
    which implies  
    $$ \psOm{ \hat{p}_{\sa}(\A) \vinit}{\POm w} = 0 $$
    using Lemma~\ref{eq:self-adj} above. Yet $\POm w = w$ since $w \in \K$, and since it holds for all $ w \in \K$ then $ \hat{p}_{\sa}(\A)\vinit \Omperp \K  $. Since $\hat{p}_{\sa}$ is a characteristic polynomial of degree $\sa$, it satisfies the monic property, that is the coefficient of $X^{\sa}$ is $1$. Writing $$\hat{p}_{\sa}(X) = X^{\sa} - \hat{q}(X),$$ 
    with $\deg(\hat{q}) \leq \sa - 1$, we obtain that $\A^{\sa} \vinit - \hat{q}(A)\vinit \Omperp \K$.
    Note that $\hat{q}(A)\vinit \in \K$ since every $z \in \K$ is equivalently written as $z = q(\A)\vinit$ with $\deg(q)\leq \sa -1$, implying 
    \begin{align*}
        \hat{q}(A)\vinit = \POm \A^{\sa} \vinit = \arg \min_{w \in \K} \norm{\Om (\A^{\sa} \vinit - w)}  \implies & \hat{q} = \arg \min_{q, \deg(q)\leq \sa -1} \norm{\Om (\A^{\sa} \vinit - q(\A)\vinit)} \\
         \implies & \hat{p}_{\sa} = \arg \min_{p \in \mathcal{PM}_\sa }  \norm{\Om p(\A)\vinit}.
    \end{align*}
\end{proof}
In the following we consider an important quantity for the convergence of the Arnoldi process,  $\norm*{(I-\Pkr)x}$, the distance of $x$ to the Krylov subspace, for $x \in \Rn{n}$ and $\Pkr$ the orthogonal projector onto $\K$. This quantity was emphasized by Youssef Saad in \cite{saad2011numerical}. When considering eigenvectors of $\A$, $x = \eivec_1$ for instance, an ideal scenario would be $\norm*{(I-\Pkr)\eivec_1} = 0$ so that we are guaranteed to find this exact eigenvector when using the Rayleigh-Ritz method with the subspace $\K$. The following inequalities, derived in \cite{Balabanov2021RandomizedblockGram}, allow us to link the oblique projection $\POm$ arising from \rArno{} to the orthogonal one $\Pkr$ in order to derive bounds on $\norm*{(I-\POm)x}$ in the following propositions. Let $x \in \Rn{n}$ and suppose $\Om$ is an $\epsilon$-embedding for $\Span{\K,x}$. Then
\begin{align*}
    \norm*{(I-\Pkr)x} & \geq (1+\epsilon)^{-1/2} \norm{\Om (I-\Pkr)x} & \quad \text{by $\epsilon$-embedding property} \\
    & \geq  (1+\epsilon)^{-1/2} \norm{\Om (I-\POm)x} & \quad \text{by definition of $\POm$} \\
    & \geq (1-\epsilon)^{1/2}(1+\epsilon)^{-1/2} \norm*{(I-\POm)x} & \quad \text{by $\epsilon$-embedding property}
\end{align*}
In the meantime, the relation $ \norm*{(I-\POm)x} \geq \norm*{(I-\Pkr)x}$ also holds since $\Pkr$ is the orthogonal projector onto $\K$. This leads to the relation 
\begin{equation}
    \norm*{(I-\Pkr)x} \leq \norm*{(I-\POm)x} \leq \sqrt{\frac{1+\epsilon}{1-\epsilon}} \norm*{(I-\Pkr)x},
    \label{eq:POmleqPkr}
\end{equation}
that bounds the oblique projector with respect to the orthogonal projector on the Krylov subspace. We use this relation in the following to study properties of randomized Arnoldi by using established bounds in the literature for $\norm*{(I-\Pkr)x}$.
We start by considering $\eivec_i$, a given eigenvector of $\A$ for some $i \in [1,n]$ and extend a result from Lemma 6.2 in \cite{saad2011numerical} to randomized Arnoldi.
\begin{proposition}
    Assume that $\A$ is diagonalizable and consider one of its eigenvectors $\eivec_i$ for some $i \in [1,n]$. Suppose that the initial vector $\vinit$ has the expression $\vinit = \sum_{j=1}^{n} \alpha_j \eivec_j$ with respect to the eigenbasis $\{\eivec_j\}$ and that $\alpha_i \neq 0$. If $\Om$ is an $\epsilon$-embedding for $\Span{\K, u_i}$, then
    \begin{equation}
        \label{eq:saadDistance}
       \norm*{(I-\POm) \eivec_i} \leq \sqrt{\frac{1+\epsilon}{1-\epsilon}} \xi_i \gamma_i^{(\sa)} 
    \end{equation}
    where 
    \begin{equation}
       \xi_i:= \sum_{j=1, j \neq i}^{n} \frac{\abs*{\alpha_j}}{\abs*{\alpha_i}}, \quad \quad \gamma_i^{(\sa)}:= \min_{p \in \mathbb{P}^*_{\sa-1}} \max_{\eival \in \Lambda (A)-\eival_i} \abs*{p(\lambda)}, 
    \end{equation}
    and $\mathbb{P}^*_{\sa-1}$ is the set of all polynomials $p$ of degree $\leq \sa-1$ such that $p(\eival_i) = 1$.
\end{proposition}
\begin{proof}
    It is proven in \cite{saad2011numerical}[Lemma 6.2] using same hypotheses that $\norm*{(I-\Pkr) \eivec_i} \leq \xi_i \gamma_i^{(\sa)}$ holds. Combined with the bounds of Equation~\eqref{eq:POmleqPkr}, the desired result is obtained.
\end{proof}

As explained in \cite{saad2011numerical} in the deterministic setting where this bound holds simultaneously for all eigenvectors of $A$, this result indicates that the Arnoldi process allows to approximate well the eigenvalues that lie at the extremes of the spectrum of $\A$, since the $\gamma_i^{(\sa)}$ term is smaller in this case. In the \rArno{} setting, since we consider $\Om$ to be an $\epsilon$-embedding for the space $\Span{\K,\eivec_i}$, or more generally an OSE for any subspace of dimension $\sa+1$ w.h.p.,  the bound in Equation~\eqref{eq:saadDistance} holds for a particular eigenvector $\eivec_i$.
It does not hold simultaneously for all eigenvectors of $A$, since this would require an $\epsilon$-embedding $\Om$ for the subspace $\Span{\K,\eivec_1,\dots,\eivec_n}$, which is of dimension $n$. However, we can still interpret that an eigenvector corresponding to an extreme eigenvalue has a smaller distance to the Krylov subspace in the sense of the oblique projection and is thus better approximated by \rArno{}.

A bound such as Equation~\eqref{eq:saadDistance} requires the decomposition of $\vinit$ in the eigenbasis which can be severely ill-conditioned when $\A$ is non-symmetric, and this can lead to large coefficients $\alpha_i$. In the following we discuss a bound on $\norm*{(I-\POm) \eivec_i}$ which shows its dependence on the conditioning of the eigenbasis $U$ and then a bound on $\norm*{(I-\POm) q_1}$  where $q_1$ is a Schur vector, which is then used to obtain a restrictive convergence result.
This corresponds to showing that two results from \cite{Bellalij2010FurtherAnalysisArnoldi}[Theorem 4.3] and \cite{Bellalij2016distanceeigenvectorKrylov}[Theorem 9] hold for randomized Arnoldi, modulo a factor $\sqrt{\frac{1+\epsilon}{1-\epsilon}}$ .
\begin{proposition}
    Assume $\A$ diagonalizable with $\A = U D U^{-1}$ the eigendecomposition and the eigenvector $u_1 \coloneqq U(:,1)$ corresponds to a distinct eigenvalue. Assume $\Om$ is an $\epsilon$-embedding for $\Span{\K, u_1}$. If $U^{-1} \vinit(1:\sa+1)$ has no zero components then
    \begin{equation}
        \sigma_{\min}(U)^2 \frac{N_b}{D_b} \leq \norm*{(I-\POm)\eivec_1}^2 \leq \frac{1+\epsilon}{1-\epsilon} \sigma_{\max}(U)^2 \frac{N_b}{D_b},
    \end{equation}
    where $N_b,D_b$ are products of distances between eigenvalues $\abs{\eival_i - \eival_j}$ multiplied by components of $\abs{U^{-1} \vinit}$.
\end{proposition}
\begin{proof}
    Under these assumptions, it is derived in \cite{Bellalij2016distanceeigenvectorKrylov} that $\sigma_{\min}(U)^2 \frac{N_b}{D_b} \leq \norm*{(I-\Pkr)\eivec_1} \leq  \sigma_{\max}(U)^2 \frac{N_b}{D_b}$. The bound for the oblique projector $\POm$ is obtained from Equation~\eqref{eq:POmleqPkr}.
\end{proof}

\begin{proposition}
Let $q_1$ be the first Schur vector from the decomposition $\A Q = Q R$ and $w_1$ the left eigenvector of $\A$ with eigenvalue $\eival_1$, that is $w_1^T \A = \eival_1 w_1^T$. Define $P_1 \coloneqq I - q_1 q_1^T$ and $B_1 \coloneqq P_1 \A P_1$. Assume $\Om$ is an $\epsilon$-embedding for $\Span{\K, q_1}$. If $\cos(w_1,\vinit) \neq 0$, then 
\begin{equation}
    \label{eq:schurDist}
   \norm*{(I-\POm) q_1} \leq \sqrt{\frac{1+\epsilon}{1-\epsilon}} \frac{\eta_{\sa}}{\abs{\cos(w_1,\vinit)}}, 
\end{equation}
with 
\begin{equation}
   \eta_{\sa} = \min_{p \in \mathbb{P}^*_{\sa-1}} \norm*{p(B_1) P_1 \vinit}, 
\end{equation}    
where $\mathbb{P}^*_{\sa-1}$ is the set of all polynomials $p$ of degree smaller or equal than $\sa-1$ such that $p(\eival_1) = 1.$
\end{proposition}
\begin{proof}
    In the setting of the theorem, it is proven in  \cite{Bellalij2010FurtherAnalysisArnoldi} that $\norm*{(I-\Pkr) q_1} \leq \frac{\eta_{\sa}}{\abs{\cos(w_1,\vinit)}} $. Again, using Equation~\eqref{eq:POmleqPkr} concludes this proof.
\end{proof}
We recall that the Schur decomposition can be ordered such that that the eigenpair $(q_1, \eival_1)$ is any eigenpair of $\A$, and thus this analysis is not restricted to a given eigenpair. It is emphasized in \cite{Bellalij2010FurtherAnalysisArnoldi} that this result allows to analyze the convergence of the Arnoldi process by using Schur vectors instead of eigenvectors.
In particular it is shown that $\eta_\sa$ corresponds to the $(\sa - 1)$th residual norm obtained when solving with GMRES the linear system 
$\tilde{\A}_1 x = \tilde{v}$, where $\tilde{\A}_1 = P_1(\A - \eival_1 I)_{|q_1^\perp}$ and $\tilde{v} = P_1 \vinit$. This allows to study the decrease of $\eta_\sa$ by using the theory of GMRES. For instance, $\eta_\sa$ tends to $0$ when all eigenvalues except $\eival_1$ lie in a disk, see \cite{Bellalij2010FurtherAnalysisArnoldi}[Section 4] for further details. In summary, the \rArno{} procedure preserves established deterministic bounds on the distance to the Krylov subspace of the form $\norm*{(I-\POm)x}$ up to a $\sqrt{\frac{1+\epsilon}{1-\epsilon}}$ factor. These results show that we can expect \rArno{} to perform similarly to deterministic Arnoldi in practice.

\subsection{Analysis of randomized Implicitly Restarted Arnoldi} 

We start by analyzing in the following lemma the error that can be added when restarting a \rArno{} procedure using Ritz information.  This error has been analyzed in the deterministic case in~\cite{Morgan1996restartingArnoldimethod}.
\begin{lemma}
    \label{th:wrong-restart}
    Let $\A \V = \V \Hsa + h_{\sa+1,\sa} v_{\sa+1} e_\sa^T$ be obtained after $\sa$ steps of the \rArno{} procedure as in Definition \ref{def:rArno}, with $\Om$ being an $\epsilon$-embedding for $\Span{\K, v_{\sa+1}}$. We note the Ritz vectors $\rivec_i \coloneqq \V \yrivec_i$, where $\yrivec_i$ is an eigenvector of $\Hsa$.  Let $\rivec_1, \rivec_2$ be the Ritz vectors corresponding to the Ritz eigenvalues $(\rival_1,\rival_2)$. Assume the procedure is restarted using
    \begin{equation}
        \label{eq:vinitWrong}
        \vinitj = \rivec_1 + \alpha \rivec_2,
    \end{equation}
    with the parameter $\alpha$ to be tuned. Define 
    \begin{equation}
        \beta_{\sa i} \coloneqq h_{\sa+1,\sa} e_{\sa}^T \yrivec_i \text{ and } \delta \coloneqq \beta_{\sa 1} + \alpha \beta_{\sa 2},
    \end{equation} 
    the related cumulative residual error as appearing in Equation~\eqref{eq:r-residualNorm}. Then after one iteration of the new \rArno{} procedure, any vector $v \in \Span{\vinitj,\A \vinitj } $ satisfies 
    \begin{equation}
      v = \sigma (\rivec_1 + e)  
    \end{equation}
    with
    \begin{equation}
       \norm*{e}^2 \geq  \frac{1}{(1+\epsilon)} \cdot   \frac{(\alpha \delta)^2}{(\delta^2 + \alpha^2(\rival_1 - \rival_2)^2)} 
    \end{equation}
    and $e \in \Span{\rivec_2, v_{\sa+1}}$.
\end{lemma}

\begin{proof}
    We follow the proof from \cite{Morgan1996restartingArnoldimethod}[Theorem 1] while taking into account that the basis is $\Om$-orthonormal. We multiply the choice of $\vinitj$ from Equation~\eqref{eq:vinitWrong} by $\A$ to get 
    $$\A \vinitj = \rival_1 \rivec_1 + \alpha \rival_2 \rivec_2 + \delta v_{\sa+1}$$
    using the \rArno{} factorization. Then for all $\gamma \in \mathbf{R}$, 
    $$\A \vinitj - \gamma \vinitj = (\rival_1 - \gamma) \rivec_1 + \alpha (\rival_2 - \gamma) \rivec_2 + \delta v_{\sa+1}.$$
    Now if a vector $v$ is in $\Span{\vinitj,\A \vinitj }$, there exists $\gamma,\sigma \in \mathbb{R}$ such that
    \begin{align*}
        v & = \frac{\sigma}{(\rival_1 - \gamma)}(\A \vinitj - \gamma \vinitj) \\
        & = \sigma \left( \rivec_1 + \alpha \frac{(\rival_2 - \gamma)}{(\rival_1 - \gamma)} \rivec_2  + \frac{\delta}{(\rival_1 - \gamma)}  v_{\sa+1} \right) \\ 
        & = \sigma (\rivec_1 + e).
    \end{align*}
     This gives the first part of the result.
     Let us now focus on the norm of $e \in \Span{\rivec_2, v_{\sa+1}}$. Using $\Om$-orthonormality of $\rivec_2$ and $v_{\sa+1}$ leads to
    $$ \norm{\Om e}^2 = \alpha^2 \frac{(\rival_2 - \gamma)^2}{(\rival_1 - \gamma)^2}  + \frac{\delta^2}{(\rival_1 - \gamma)^2}.$$ 
    To find a lower bound for $\norm{\Om e}$, we minimize it as a function of $\gamma$ and obtain
    $$\gamma^* \coloneqq \arg \min_{\gamma} \norm{\Om e}^2 = \rival_2 + \frac{\delta^2}{\alpha^2 (\rival_2 - \rival_1)}.$$
    Substituting $\gamma^*$ in $\norm{\Om e}^2$ and using the $\epsilon$-embedding property concludes the proof. 
\end{proof}
This result shows that, similarly to the deterministic case, restarting in the randomized framework using a linear combination of Ritz information can add errors to $\rivec_1$, where $\rivec_1$ should approximate $\eivec_1$ in the newly generated subspace. Indeed, $\rivec_1$ is not solely contained in $\Span{\vinitj,\A \vinitj }$ since there is always an additive term $e$. Following Morgan reasoning in \cite{Morgan1996restartingArnoldimethod},  suppose a strong eigenvalue gap such that $\delta \ll \abs{\rival_1 - \rival_2}$ and that $\rivec_1$ is a good approximation to the eigenvector $\eivec_1$ such that $\beta_{\sa 2} \gg \beta_{\sa 1}$. Then after the restart an error $\norm{e} \propto \beta_{\sa 2}$ was added to $\rivec_1$,  which is not desirable. It is shown in \cite{Morgan1996restartingArnoldimethod} that the only choice that does not add error in this situation is $\alpha \coloneqq -\frac{\beta_{\sa 1}}{\beta_{\sa 2}}$, and that is what IRA does implicitly. We show later in Proposition~\ref{prop:spansrestart} that rIRA does not add any error through the restarts. To continue our analysis, first we derive a key result that relates any \rArno{} factorizations of a given matrix $A$.
It shows that any two \rArno{} factorizations that have been established starting from the same initial vector are identical, thus a \rArno{} decomposition is uniquely defined by $\V e_1 = \vinit$.

\begin{theorem}[Randomized implicit Q-theorem]
    \label{th:r-implicit-Q}
    Consider two \rArno{} factorizations obtained after $\sa$-steps of the \rArno{} procedure as in Definition \ref{def:rArno}
    \begin{align*}
        AV = VH + re_k^T, \\
        AQ = QG + fe_k^T,
    \end{align*}
    where $\Om$ is an $\epsilon$-embedding for $\Span{V,r,Q,f}$, the matrices $V,Q$ are $\Om$-orthonormal and $H,G$ are upper Hessenberg with positive sub-diagonal elements. If the starting vectors are equal, i.e.\@ $Ve_1 = Qe_1$, and both $r \Omperp V$ and $f \Omperp Q$, 
    then $V = Q, H = G$, and $r = f$. equivalently, the decompositions are the same.
\end{theorem}
\begin{proof}
    We proceed with an induction on $k$, the number of columns of $Q,V$.
    For $k = 1$, we have 
    \begin{align*}
        Av = h  v + r, \\
        Aq = g  q + f,
    \end{align*}
    with $h,g$ scalars. The hypothesis $Ve_1 = Qe_1$ directly gives $v = q$, and by substracting these two equations we obtain
    $$ 0 = (h-g)v + (r-f).$$
    Multiplying by $(\Om v)^T \Om$ and using the sketch-orthogonal hypotheses $(\Om v)^T \Om r = (\Om v)^T \Om f = 0$ shows that $0 = h - g$, thus $h = g$ and $r = f$.
    Assume that the result holds for $k$ and that we are given factorizations of dimension $k+1$. Decompose $V,Q$ and $H,G$ such that $A[\tilde{V} \; v] = [\tilde{V} \; v]
    \begin{pmatrix}
        \tilde{H} & h \\
        \beta e_k^T & \alpha \\ 
    \end{pmatrix} 
    + r e_{k+1}^T$  and 
    $A[\tilde{Q} \; q] = [\tilde{Q} \; q]
    \begin{pmatrix}
        \tilde{G} & g \\
        \gamma e_k^T & \delta \\ 
    \end{pmatrix} 
    + f e_{k+1}^T$
    thanks to the Hessenberg structure. This gives
    \[
    \left\{ \begin{array}{c}
        AV = VH + re_k^T \\
        AQ = QG + fe_k^T 
    \end{array} \right.
    \implies
    \left\{ \begin{array}{c}
        A[\tilde{V}, v] = [\tilde{V}\tilde{H} + v \beta e_k^T, \tilde{V} h + \alpha v + r] \\
        A[\tilde{Q}, q] = [\tilde{Q}\tilde{G} + q \gamma e_k^T, \tilde{Q} g + \delta q + f]
    \end{array}, \right.
    \]
    which translates to 
    \[
    \left\{ \begin{array}{c}
        A \tilde{V} = \tilde{V} \tilde{H} + v \beta e_k^T \\
        A \tilde{Q} = \tilde{Q} \tilde{G} + v \gamma e_k^T 
    \end{array} \right.
    \text{and}
    \left\{ \begin{array}{c}
        Av = \tilde{V}h + \alpha v + r \\
        Av = \tilde{Q}g + \delta q + f
    \end{array}. \right.
    \]
    The left system of equations is of dimension $k$ and satisfies the induction hypothesis, while the right system of equations can be developed similarly to the case $k=1$.
\end{proof}
We have shown in section~\ref{sec:rIRA} that after truncation in Algorithm~\ref{alg:rIRA}, a new legitimate \rArno{} factorization is obtained. The above result adds that this factorization is the same as the one obtained by starting from $\vinitj \coloneqq \V \Q e_1$, where we recall that $\Q$ is the orthogonal matrix obtained from the QR shifted steps on $\Hba$. Let us specify the expression for the new starting vector $\vinitj$ in rIRA with respect to the previous $\vinit$. We summarize here the rIRA scheme:
\begin{enumerate}
    \item Obtain a length $\ba$ \rArno{} factorization $\A \V = \V \Hba + \rba e_{\ba}^T$.
    \item Apply the shifted QR algorithm to $\Hba$ to obtain the unitary similarity transformation $ \Q^T \Hba \Q $.
    \item Truncate the resulting $\A (\V \Q) = (\V \Q)  \Q^T \Hba \Q + \rba e_{\ba}^T \Q $ to its first $\sa$ vectors to obtain a length $\sa$ \rArno{} decomposition and go back to 1.
    
\end{enumerate}
It is important to note that properties of the shifted QR algorithm applied to a Hessenberg matrix $\Hba$  have been derived in works such as \cite{Lehoucq1995DeflationTechniquesImplicitly} or \cite{Miminis1991ImplicitShiftingQR} and can be used here since this part is not modified in our randomized approach. Suppose the shifts $\shi_i, i = 1,\dots, \nshi$ are used, defining the filtering polynomial $\polp(\lambda) = \prod_{i = 1}^{\nshi}(\lambda - \shi_i)$. We then have from \cite[Theorem 3.2]{Lehoucq1995DeflationTechniquesImplicitly} a relation between the underlying filtering polynomial $\polp$ and the matrix $\Q$,
\begin{equation}
   \Q e_1 = \polp(\Hba) \rho e_1, 
\end{equation}
where $\rho$ is a normalisation constant. In addition, \cite[Lemma 4.1]{Lehoucq1995DeflationTechniquesImplicitly} 
states that, for a polynomial $\polp(\A)$ and a decomposition of the form $\A \V = \V \Hba + \rba e_{\ba}^T$, where $\Hba$ is Hessenberg and unreduced, and $\V$ does not have to be orthogonal (hence can be $\Om$-orthonormal),
\begin{equation}
  \polp(\A) \V [e_1, \dots, e_{\sa}] = \V \polp(\Hba) [e_1, \dots, e_{\sa}].  
\end{equation}
In other words, even though $\V \in \Rnm{n}{(\ba)}$ is not a true invariant subspace for $\A$, given that $\rba \neq 0$, one can still have a commutation property for a polynomial of degree at most $\nshi$. The two previous results combined lead to
\begin{align}
    \vinitj & \coloneqq \V \Q e_1 \nonumber \\
    & = \V \polp(\Hba) \rho e_1 \nonumber \\
    & = \rho \polp(\A) \V e_1 \nonumber \\
    & =  \rho \polp(\A) \vinit. \label{eq:vinitj-polpA}
\end{align}
Our algorithm rIRA is thus achieving the same goal as the deterministic IRA, namely it applies implicitly a filtering polynomial on $\A$ to the starting vector. 
Using the randomized implicit Q theorem above, we can state that computing a new \rArno{} process starting from $\vinitj = \rho \polp(\A) \vinit$ is equivalent to obtaining the decomposition after shifted QR and truncation. Indeed, in both cases we have $\vinitj = \rho \polp(\A) \vinit$ and thus the factorizations are equal. Moreover, the subspace spanned by the columns of $\V \Q$ after truncation is  
\begin{equation}
        \Span{\V \Q [e_1, \dots, e_{\sa}]} = \K(\A,\rho \polp(\A) \vinit).
\end{equation}
We now discuss the consequences of the exact shifts selection strategy briefly mentioned earlier, especially on $\vinitj$. Let $\A \V = \V \Hba + \rba e_{\ba}^T$ be an unreduced length $\ba$ \rArno{} factorization as in Definition~\ref{def:rArno}. Denote the spectrum of $\Hba$ as $\Lambda(\Hba) = \{\rival_1,\dots,\rival_{\sa}\} \cup \{\rival_{\sa+1},\dots,\rival_{\ba}\}$ with an arbitrary separation, with potential complex conjugate pairs in the same subset. The exact shifts selection is defined as 
\begin{equation}
    \shi_i \coloneqq \rival_i, \; i = \sa +1, \dots, \ba.
\end{equation}
Then \cite[Theorem 4.4]{Lehoucq1995DeflationTechniquesImplicitly} states that the shifted QR steps produce
\begin{equation}
    \Q^T \Hba \Q = \begin{pmatrix}
        H_{11} & M \\
        0 & H_{12}
    \end{pmatrix}, 
\end{equation}
where $\Lambda(H_{11}) = \{\rival_1,\dots,\rival_{\sa}\}$ and $\Lambda(H_{12}) = \{\rival_{\sa+1},\dots,\rival_{\ba}\}$. Most importantly the procedure has annihilated the $(\sa+1,\sa)$ entry of $\Hba$.   
In addition, given the $\sa \times \sa$ square orthogonal Schur factorization 
\begin{equation}
  H_{11} Z_1 = Z_1 R_1  
\end{equation}
and defining 
\begin{equation}
  Q_1 \coloneqq \Q [e_1, \dots, e_{\sa}],  
\end{equation}
we have that
\begin{equation}
    \vinitj = \V \Q e_1 = \V Q_1 e_1 = \V Q_1 Z_1 Z_1^T e_1\in \Span{\V Q_1 Z_1}.
\end{equation}
The two main takeaways from this discussion are:
\begin{itemize}
    \item Since the updated Hessenberg matrix $\Hsaplus$ of line~\ref{alg:rIRA-truncate} of Algorithm~\ref{alg:rIRA} is defined as $\Hsaplus \coloneqq H_{11}$, and the eigenvalues of $\Hsaplus$ are the Ritz values of $\A$, it can be seen that the exact shifts allow to discard the unwanted eigenvalues and thus target a specific part of the spectrum of $A$. This in-built capacity of rIRA differentiates it from many other eigensolvers that require some preconditioning to do so such as shift and invert, see how in \cite{Parlett1987Complexshiftinvert} and why in \cite{Scott1982AdvantagesInvertedOperators} for instance.
    \item $(\V Q_1 Z_1)$ are approximate $\Om$-orthonormal Schur vectors of $\A$. Indeed, since $\A (\V Q_1) = (\V Q_1) H_{11} + \tilde{\rsa} e_{\sa}^T $ after truncation, one has 
    \begin{equation}
    \label{eq:approxOmSchur}
       \A (\V Q_1 Z_1) = (\V Q_1 Z_1) R_1  + \tilde{\rsa} e_{\sa}^T Z_1 \approx (\V Q_1 Z_1) R_1  
    \end{equation}
    with $(\V Q_1 Z_1)$ being $\Om$-orthogonal. Since $\vinitj \in \Span{\V Q_1 Z_1}$, this means that the starting vector is close to a linear combination of $\Om$-orthonormal Schur vectors of $\A$.  In other words, we achieved one of the main goals of the method, motivated by Theorem~\ref{th:schur-invariant}, which makes $\rsa$ tend to $0$.
\end{itemize}

We now address the question of potential errors added to the Krylov subspace when restarting the factorization, as stated in Lemma~\ref{th:wrong-restart}. To do so, we show that the subspace generated by $\V \Q [e_1,\dots, e_{\sa}] = \V Q_1$ spans exactly the previous Ritz vectors $\rivec, \dots, \rivec_\sa$.
\begin{proposition}
    \label{prop:spansrestart}
    Consider a \rArno{} factorization $\A \V = \V \Hba + \rba e_{\ba}^T$ as in Definition \ref{def:rArno} with wanted Ritz values $\rival_1,\dots,\rival_\sa$ and associated Ritz vectors $\rivec_1,\dots,\rivec_\sa$. After the shifted QR steps and the truncation have been applied as in algorithm rIRA~\ref{alg:IRA}, the new Krylov basis $\V Q_1 \in \Rnm{n}{\sa}$ satisfies 
    \begin{equation}
    \Span{\V \Q_1} = \Span{\rivec_1,\dots,\rivec_{\sa}}.
    \end{equation}
\end{proposition}

\begin{proof}
We recall the relation between $\polp$ and $\Q$ from \cite[Theorem 3.2]{Lehoucq1995DeflationTechniquesImplicitly},
$$ \Q T = \polp(\Hba) $$
where $T = R^p \dots R^1$ is the product of all upper triangular factors obtained when computing QR factorizations on line~\ref{alg:rIRA-qrfacto}. When $\nshi$ exact shifts are applied to a Hessenberg matrix, we have from \cite[Theorem 4.1]{Miminis1991ImplicitShiftingQR} that $T$ is upper triangular with $k$ non-zero first  diagonal entries and $p$ zeros last diagonal entries. Thus
$$\Q T e_1 = \Q t_{11} e_1 = \polp(\Hba) e_1 \implies q_1 \coloneqq Qe_1 =  \polp(\Hba) w $$
for some vector $w \coloneqq e_1 / t_{11} \in \Rn{\ba}$. If one writes $w = \sum_{i = 1}^{\ba} \alpha_i \yrivec_i$ its decomposition in the eigenbasis of $\Hba$ whose existence is assumed, then 
$$ q_1 = \sum_{i = 1}^{\ba} \alpha_i \polp(\Hba) \yrivec_i = \sum_{i = 1}^{\ba} \alpha_i \polp(\rival_i) \yrivec_i = \sum_{i = 1}^{\sa} \alpha_i \polp(\rival_i) \yrivec_i \in \Span{\yrivec_1,\dots,\yrivec_{\sa}}.$$
Thus 
\begin{equation}
    \vinitj \coloneqq \V q_1 = \sum_{i = 1}^{\sa} \alpha_i \polp(\rival_i) \rivec_i \in \Span{\rivec_1,\dots,\rivec_{\sa}}.
\end{equation}
This property emphasizes that we are indeed restarting using the Ritz information from the previous iteration, and that the strategy of exact shifts has the advantage of cancelling out the components along the eigenvectors related to the unwanted eigenvalues $\{\rival_{\sa+1},\dots,\rival_{\ba}\}$. A quick induction on $q_i = QT e_i = Q (t_{11} e_1 + \dots +t_{ii} e_i ) = \polp(\Hba) e_i $ that uses the structure of $T$ explained above gives
$$ \V \Span{q_1,\dots,q_{\sa}} \subset \V \Span{\yrivec_1,\dots,\yrivec_{\sa}}.$$
Finally, we know from corollary \ref{cor:rando_sval} that $\V$ is well conditioned and will let linearly independent any set of vectors on which it is applied, obtaining
$$\Span{\V \Q [e_1,\dots, e_{\sa}]} = \Span{\rivec_1,\dots,\rivec_{\sa}}.$$
\end{proof}
This proves that using exact shifts and truncation has the effect to implicitly restart the \rArno{} factorization with a new Krylov subspace of dimension $\sa$ that contains exactly the previous wanted Ritz vectors without any additional error added, addressing the concerns of Lemma~\ref{th:wrong-restart}.

\subsection{A convergence result for randomized Implicitly Restarted Arnoldi} 

To analyze the convergence of rIRA, we first discuss the dependence of $\rsa$ as a function of $\vinit$ for a given \rArno{} factorization 
\begin{equation}
    \label{eq:rArno-cv}
    \A \V = \V \Hsa + \rsa e_{\sa}^T
\end{equation} with $\Hsa$ unreduced as in Definition \ref{def:rArno}. For this, similarly to the discussion that led to \cite[Theorem 2.7]{Sorensen1992ImplicitApplicationPolynomial}, we define 
$$ K \coloneqq [\vinit, \dots, \A^{\sa-1} \vinit] \in \Rnm{n}{\sa} \quad \text{and} \quad F \coloneqq  
\begin{bmatrix}
    0 &  &  & &  \\
    1 & & & & \\
    & 1 & & , & g \\
    & & \ddots & & \\
    & & & 1 & \\
\end{bmatrix} \in \Rnm{\sa}{\sa},$$
the Krylov matrix $K$, which is full rank since $\Hsa$ is unreduced, and the associated companion matrix $F$, with $ g \in \Rn{\sa}$ a vector that we are free to specify later. These objects satisfy:
\begin{itemize}
    \item A factorization of $\A$ which writes
        \begin{equation}
        \label{eq:AKKF}
            \A K = K F + t e_{\sa}^T
        \end{equation} 
    where $ t \coloneqq \A^{\sa} \vinit - Kg$,
    \item $\hat{p}(\lambda) = \lambda^{\sa} + \ps{[\lambda^0 \dots \; \lambda^{\sa-1}]}{g} $ is the characteristic polynomial of $F$,
    \item $t = \hat{p}(\A) \vinit$.
\end{itemize}
Consider an $\epsilon$-embedding $\Om$ for $\Span{\K,\A^{\sa} \vinit}$ and choose $g$ such that it minimizes the sketched norm of the residual, i.e.\@ 
\begin{equation}
    \label{eq:g-Om-lsq}
    g = \arg \min_{w \in \Rn{\sa}} \norm*{\Om(\A^{\sa} \vinit - K w)}. 
\end{equation}
To obtain its expression, consider the randomized QR decomposition of $K$, with Algorithm~\ref{alg:sketch-orthonormal} for instance, $K = QR$ where $Q \in \Rnm{n}{\sa}$ is $\Om$-orthonormal and $R$ upper triangular with strictly positive diagonal elements. Then the solution to the least squares problem~\eqref{eq:g-Om-lsq} is 
$$ g = [(\Om K)^T (\Om K)]^{-1} (\Om K)^T \Om \A^{\sa} \vinit = [R^T R]^{-1} R^T (\Om Q)^T \Om \A^{\sa} \vinit \implies g = R^{-1}  (\Om Q)^T \Om \A^{\sa} \vinit.$$
This leads to $t = \A^{\sa} \vinit -  Q (\Om Q)^T \Om \A^{\sa} \vinit$ such that 
$$ (\Om K)^T \Om t = R^T  (\Om Q)^T \Om \A^{\sa} \vinit - R^T  (\Om Q)^T \Om \A^{\sa} \vinit = 0 \implies t \Omperp \K, $$
i.e.\@ $t$ is $\Om$-orthogonal to the Krylov subspace $\K$, hence $\Om$-orthogonal to $Q$. We now work on Equation~\eqref{eq:AKKF} to obtain a \rArno{} factorization. Multiplying it by $R^{-1}$ gives $\A (K R^{-1}) = (K R^{-1})R F R^{-1} + t e_{\sa}^T R^{-1}$, or equivalently
\begin{equation}
    \label{eq:AQQG}
    \A Q = Q G + f e_{\sa}^T 
\end{equation}
with $Q = K R^{-1}$, $G \coloneqq R F R^{-1}$ and $f \coloneqq t / \rho_{\sa \sa}$ where $\rho_{ii} \coloneqq e_i^T R e_i$. Note that $G$ has the same characteristic polynomial as $F$ since they are similar. We now prove that Equation~\eqref{eq:rArno-cv} and Equation~\eqref{eq:AQQG} are equal by using Theorem~\ref{th:r-implicit-Q} whose hypotheses are verified below:
\begin{itemize}
    \item $Q$ is $\Om$-orthonormal,
    \item $G$ is upper Hessenberg because $F$ is and $R,R^{-1}$ are upper triangular which does not change the Hessenberg structure when multiplied with $F$,
    \item $f \Omperp Q$ because $f$ is proportional to $t$,
    \item $Qe_1 = K R^{-1} e_1 = K e_1 \norm*{\Om \vinit}^{-1} = K e_1 = \vinit = V e_1$, using $\vinit$ of $\Om$-norm $1$ and $R^{-1} e_1 = 1/\rho_{11} = \norm*{\Om \vinit}^{-1}$ from the randomized QR decomposition. 
\end{itemize}
We thus obtain $\V = Q, \; \Hsa = G, \; \rsa = f$. This means that $\Hsa$ has $\hat{p}$ as its characteristic polynomial. A relationship between $\rsa$ and $\vinit$ is obtained by using $h_{j+1,j} = e_{j+1}^T \Hsa e_j$ and noting that according to the \rArno{} procedure $h_{j+1,j} = \norm*{\Om r_j}$ for each vector $r_j$, $1 \leq j \leq k$. Since the analysis above holds for any size of a \rArno{} factorization, we have for a given $j$,
$$ h_{j+1,j} = \norm*{\Om r_j} = \norm*{\Om f_j} = \norm*{\Om t_j / \rho_{jj}} = \norm*{\Om \hat{p}_j(\A) \vinit } / \rho_{jj}.$$
In addition,
$$ h_{j+1,j} = e_{j+1}^T R F R^{-1} e_j = \rho_{j+1,j+1} / \rho_{jj}$$
leading to  $ \rho_{j+1,j+1} = \norm*{\Om \hat{p}_j(\A) \vinit }$. This leads to the following theorem.
\begin{theorem}
    \label{th:rsa-vinit}
    Let $\A V_j = V_j H_j + r_j e_j^T$ be a sequence of j successive inner \rArno{} steps as obtained from Algorithm~\ref{alg:RCGS2-Arnoldi} with $1 \leq j \leq k$ such that $V_j \in \Rnm{n}{j}$, $H_j \in \Rnm{j}{j}$, $\Hsa$ is unreduced and $\Om$ is an $\epsilon$-embedding for $\Span{V_\sa,\rsa}$. Let $\hat{p}_j$ be the characteristic polynomial of $H_j$. Then,
    \begin{equation}
        r_j = \frac{\hat{p}_{j}(A)v_1}{\norm{\Om \hat{p}_{j-1}(A)v_1}}.
    \end{equation}
    In addition, using Theorem~\ref{th:charac-poly-min} we obtain that $\hat{p}_{j}$ minimizes $\norm*{\Om p(\A) \vinit}$ over all monic polynomial $p$ of degree $j$.
\end{theorem}
\begin{proof}
    The discussion above gives $r_j = f = t/\rho_{jj} = \frac{\hat{p}_{j}(A)v_1}{\norm{\Om \hat{p}_{j-1}(A)v_1}}$, notably for any $j \leq \sa$ where $\sa$ is given. 
\end{proof}
Note that since $r_j$ was in $\mathcal{K}_{j+1}(\A,\vinit)$, it could be already written as $r_j = p(\A) \vinit$  for some polynomial $p$ of degree $j$. The above theorem shows that $p$ is the characteristic polynomial of $H_j$, which satisfies the optimally property of Theorem~\ref{th:charac-poly-min}. 

We now present a convergence result in the general non-symmetric case that holds when the shifts are no longer exact but fixed for all the outer iterations of rIRA. Although this shifting strategy is not used in practice since less interesting than the exact one, the following theorem gives insights on the behaviour of the method when close to convergence, a moment where exact shifts can be considered as fixed shifts, as emphasized by Sorensen in \cite[Theorem 5.1]{Sorensen1992ImplicitApplicationPolynomial}.
\begin{theorem}
    \label{th:rIRAconv}
    Let $\nshi$ shifts ${\shi_1,\dots,\shi_p}$ define $\polp(\lambda) = \prod_{i=1}^{\sa}(\lambda - \shi_i)$. Assume that several rIRA iterations are performed on $\A$, indexed by $i$ and using an OSE $\Om$ which embeds every subspace of dimension up to $\ba$ with high probability as in Definition~\ref{def:OSE}. Note $\vinit^{(i=0)} \coloneqq \vinit$ and 
    $$ \vinit^{(i)} = \polp(\A) \vinit^{(i-1)} / \norm*{\Om \polp(\A) \vinit^{(i-1)} }$$ 
    the starting vectors of $\Om$-norm 1, using Equation~\eqref{eq:vinitj-polpA}. The spectrum of a matrix $\A \in \Rnm{n}{n}$ is divided as $\Lambda(\A) = \{\eival_1, \dots, \eival_{\sa} \} \cup \{\eival_{\sa+1}, \dots, \eival_{n} \}$ where $\{\eival_1, \dots, \eival_{\sa} \}$ are the eigenvalues of interest.
    If $\pi_{i} = h_{2,1}^{(i)} \dots h_{\sa+1, \sa}^{(i)}$ is the product of the subdiagonal elements of $\Hsa$ at the $i$'th rIRA iteration, that 
    $$ \abs*{\polp(\eival_1)} \geq \dots \geq \abs*{\polp(\eival_k)} > \abs*{\polp(\eival_{k+1})} \geq \dots \geq \abs*{\polp(\eival_n)} $$
    with
    $$ \gamma = \abs*{\polp(\eival_k)} / \abs*{\polp(\eival_{k+1})} < 1 $$
    and that $v_1$ is not in an invariant subspace tied to ${\lambda_{k+1}} \dots \lambda_{n}$,
    then the sequence $\{\pi_{i}\}$ converges to $0$ and there exists a constant $K$ and a positive integer $I$ such that for $i > I$,
    \begin{equation}
      0 \leq \pi_{i} \leq \sqrt{ \frac{1+\epsilon}{1-\epsilon} }\gamma^{i} K.  
    \end{equation}
\end{theorem}
\begin{proof}
    Suppose
    $\A (Q_1, Q_2) = (Q_1, Q_2) \begin{pmatrix}
        R_1 & M \\
        0 & R_1
    \end{pmatrix}$
    is a Schur decomposition of $\A$ with $\Lambda(R_1) = \{\eival_1, \dots, \eival_{\sa} \}$.  
    Using the expression of $\vinit^{(i)}$ inductively, we obtain 
    $$ \vinit^{(i)} = \polp^i (\A) \vinit / \norm*{\Om \polp^i (\A) \vinit}. $$
    Now from Theorem~\ref{th:rsa-vinit} we have that 
    $$h_{j,j+1}^{(i)} = \frac{\norm*{\Om \hat{p}_{j}(A) \vinit^{(i)}}}{\norm*{\Om \hat{p}_{j-1}(A) \vinit^{(i)}}},$$ 
    hence in the product $\pi_{i} = h_{2,1}^{(i)} \dots h_{\sa+1, \sa}^{(i)}$ many terms get cancelled, resulting in 
    $$\pi_{i} = \norm*{\Om \hat{p}_{\sa}(A) \vinit^{(i)}} / \norm*{\Om \vinit^{(i)}} = \norm*{\Om \hat{p}_{\sa}(A) \vinit^{(i)}} = \min_{p, \deg(p) \leq \sa} \norm*{\Om p(\A) \vinit^{(i)}},$$
    where the last equality comes from Theorem~\ref{th:charac-poly-min}.  Letting $q$ be the characteristic polynomial of $R_1$, which has a degree less or equal to $\sa$, we obtain
    $$\pi_{i} \leq  \norm*{\Om q(\A) \vinit^{(i)}} = \norm*{\Om q(\A)  \hspace{0.2cm} \polp^i (\A) \vinit / \norm*{\Om \polp^i (\A) \vinit}  \hspace{0.2cm} } \implies \pi_{i} \cdot \norm*{\Om \polp^i (\A) \vinit} \leq \norm*{\Om q(\A) \polp^i (\A) \vinit }.$$
The $\epsilon$-embedding property can be applied on both sides of this inequality at any iteration $i$, since $\Om$ is an OSE, hence it is an $\epsilon$-embedding for every subspaces of dimension at most $\ba$ w.h.p.. We have $\polp^i (\A) \vinit = \alpha \vinit^{(i)} \in \K(\A, \vinit^{(i)})$, which is of dimension $\sa$, and similarly $q(\A) \polp^i (\A) \vinit = q(\A) \alpha \vinit^{(i)} \in \mathcal{K}_{\sa+1} (\A, \vinit^{(i)})$, which is of dimension $\sa +1$. It results that
    $$ \pi_{i} \cdot \sqrt{1-\epsilon} \norm*{\polp^i (\A) \vinit} \leq \sqrt{1+\epsilon} \norm*{ q(\A) \polp^i (\A) \vinit }.$$ 
    Defining $\pi_{i}^{\Om} \coloneqq \sqrt{\frac{1-\epsilon}{1+\epsilon}} \pi_{i}$, we obtain  
    $$\pi_{i}^{\Om} \cdot \norm*{\polp^i (\A) \vinit} \leq  \norm*{ q(\A) \polp^i (\A) \vinit }. $$
    By doing so, we separated randomized and deterministic quantities, since $q$ and $\polp$ are fixed and independent of the method IRA or rIRA. It was shown in \cite[Theorem 5.1]{Sorensen1992ImplicitApplicationPolynomial} the existence of an integer $I$ and a positive constant $K$ such that $i > I$ implies $\pi_{i}^{\Om} \leq \gamma^{i} K$, where the norm of $\vinit$ only affects the constant values, and this concludes the proof. 
\end{proof}
This convergence result can be interpreted as having the product $\pi_i$ of the positive sub-diagonal elements of $\Hsa$ converging to zero, so that some $ h_{j+1,j} \rightarrow 0$ meaning that eventually $r_j = 0$ in the \rArno{} algorithm for $j \leq \sa$. Since $r_j = h_{j+1,j} v_{j+1}$, this results in $\V_j$ being an invariant subspace for $\A$.

\section{Numerical results}
\label{sec:numericals}
In this section we study the numerical efficiency of 
rIRA for solving non-symmetric eigenvalue problems. The experiments are conducted using Julia. 
Most of the matrices used are \href{https://sparse.tamu.edu/Grund/poli4}{from the SuiteSparse Matrix Collection} and are summarized in Table~\ref{tab:testMatrices}.

\begin{table}
    \centering
    \begin{tabular}{|l|l|l|l| }
        \hline
        \textbf{Name} & \textbf{Size n} & \textbf{Nonzeros} & \textbf{Problem} \\
        \hline
        \href{https://sparse.tamu.edu/Grund/poli4}{poli4} & $33,833$ & $73,249$ & Computational Fluid Dynamics \\
        \hline
        \href{https://sparse.tamu.edu/Norris/lung2}{lung2} & $109,460$ & $492,564$ & Computational Fluid Dynamics \\
        \hline
        \href{https://sparse.tamu.edu/CEMW/tmt_unsym}{tmt\_unsym} & $917,825$ & $4,584,801$ & Electromagnetic Problem \\
        \hline
        \href{https://sparse.tamu.edu/VLSI/vas_stokes_1M}{Vas\_stokes\_1M} & $1,090,664$ & $34,767,207$ & Semiconductor Process Problem \\
        \hline
        \href{https://sparse.tamu.edu/Hamrle/Hamrle3}{Hamrle3} & $1,447,360$ & $5,514,242$ & Circuit Simulation Problem \\
        \hline
        \href{https://sparse.tamu.edu/Bourchtein/atmosmodl}{atmosmodl} & $1,489,752$ & $10,319,760$ & Computational Fluid Dynamics \\
        \hline
        \href{https://sparse.tamu.edu/Janna/ML_Geer}{ML\_Geer} & $1,504,002$ & $110,686,677$ & Structural Problem  \\
        \hline
    \end{tabular}
    \caption{Test matrices for rIRA}
    \label{tab:testMatrices}
\end{table}

\subsection{Computation of a set of eigenvalues of interest}
We discuss first the spectrum selection property of rIRA and consider a simple case, a non-symmetric matrix of dimensions $800 \times 800$ whose spectrum is $\Lambda = \{1,2,3,\dots,800\}$. We seek $\sa = 10$ eigenvalues of interest, namely those of largest modulus (LM) and smallest modulus (SM), in a subspace of dimension $m \coloneqq \ba = 50$. In all this section, the sketching size is  $\dsk = 4m$ which here gives $\dsk= 200$. An iteration is defined as an outer iteration of rIRA, i.e.\@ Arnoldi extension, shifted QR steps, and truncation, or equivalently Line~\ref{alg:rIRA-arnoldi-extend} to Line~\ref{alg:rIRA-truncate} of Algorithm~\ref{alg:rIRA}. Figure~\ref{fig:diagLM} and Figure~\ref{fig:diagSM} display the residual norms $\norm*{\A \rivec - \rival \rivec} / \norm*{\rivec}$ reached during the iterations for each pair of interest. 
Note that these residual norms correspond to the central term in the inequalities of Equation~\eqref{eq:r-residualNorm-free}. They are too costly to calculate in a real application, but are used in this section to validate the numerical efficiency of rIRA.  However, convergence is defined using the sketched residual error, i.e.\@ the quantity $\norm*{\Om(\A \rivec - \rival \rivec)} = h_{\sa+1,\sa} \abs{e_{\sa}^T \yrivec}$ from Equation~\eqref{eq:r-residualNorm-free}, which is available for no extra cost. The algorithm stops when the maximum of these sketched residual errors over all the sought eigenpairs reaches a threshold $\eta$ defined in this experiment as $\eta = 10^{-8}$. The maximum is represented in the figure by a dotted line. 
We observed in practice that a faster convergence is obtained if in addition to the $\sa$ wanted eigenpairs,  we compute $4$ supplementary eigenpairs while maintaining a Krylov subspace of dimension $m = 50$. Indeed, the eigenvalues farthest from the LM or SM targeted eigenpairs tend to converge the slowest and worsen the overall behavior. This is also noted in \cite{Sorensen1992ImplicitApplicationPolynomial} as an \emph{ad hoc stopping rule}. Figures~\ref{fig:eivaldiagLM} and~\ref{fig:eivaldiagSM} display all the $\sa + 4$ approximated eigenvalues at each iteration. It can be seen that convergence is achieved after a few iterations for the LM and SM cases and that the computed eigenvalues approximate well the prescribed eigenvalues of $\A$.
\begin{figure}
    \centering
    \begin{subfigure}[b]{0.45\textwidth}
        \centering
        \includegraphics[width=\textwidth]{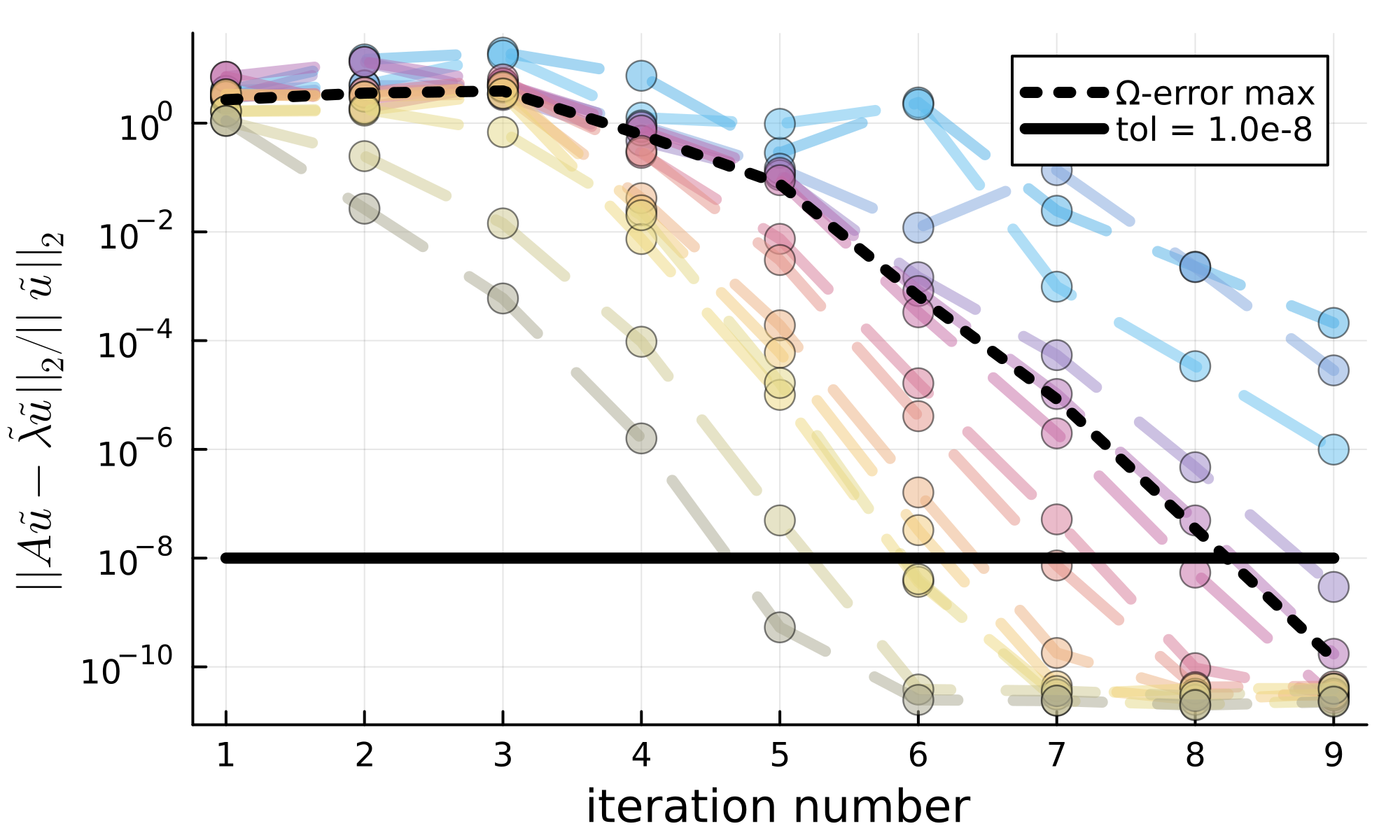}
        \caption{Residual norms largest modulus (LM)}
        \label{fig:diagLM}
    \end{subfigure}
    \hfill
    \begin{subfigure}[b]{0.45\textwidth}
        \centering
        \includegraphics[width=\textwidth]{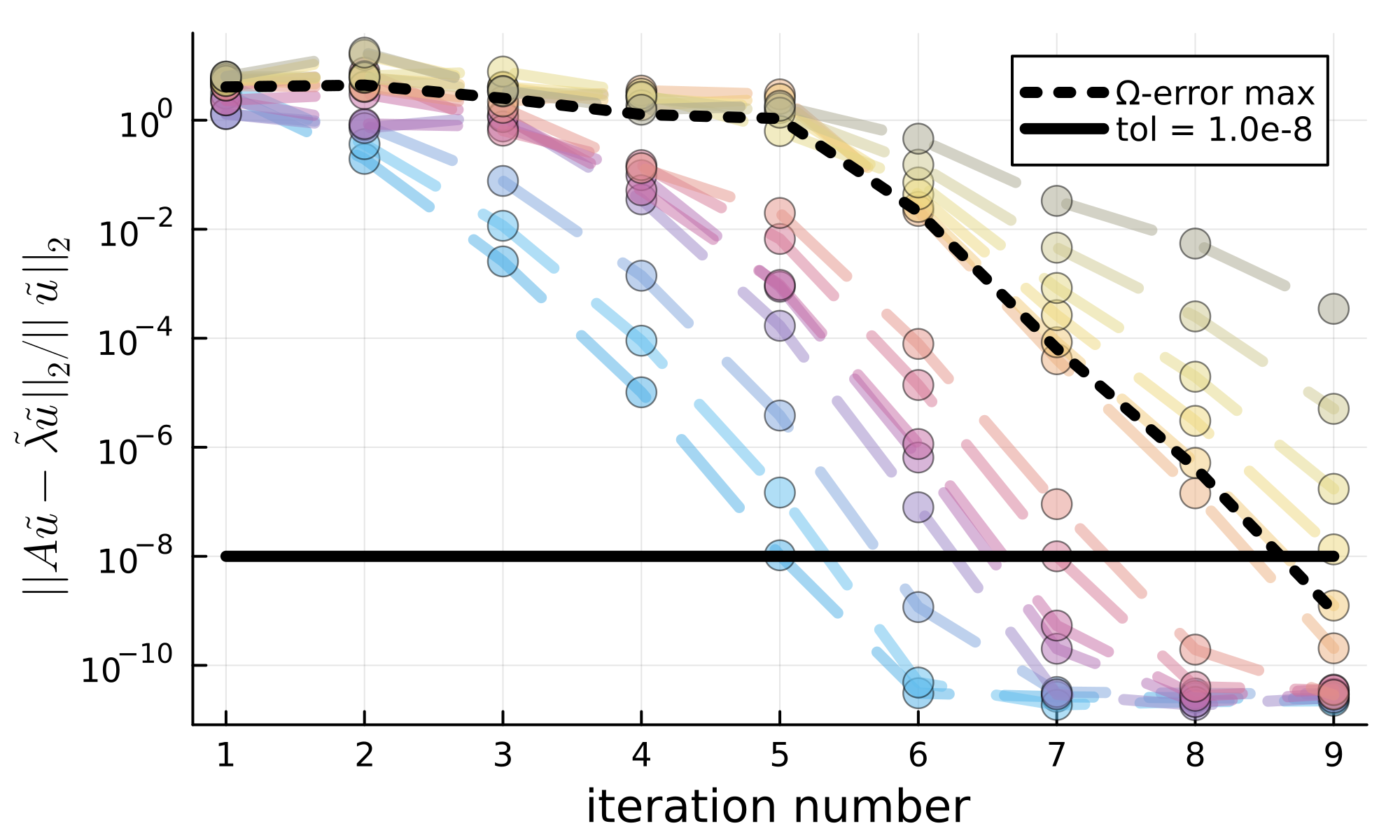}
        \caption{Residual norms smallest modulus (SM)}
        \label{fig:diagSM}
    \end{subfigure}
    \begin{subfigure}[b]{0.45\textwidth}
        \centering
        \includegraphics[width=\textwidth]{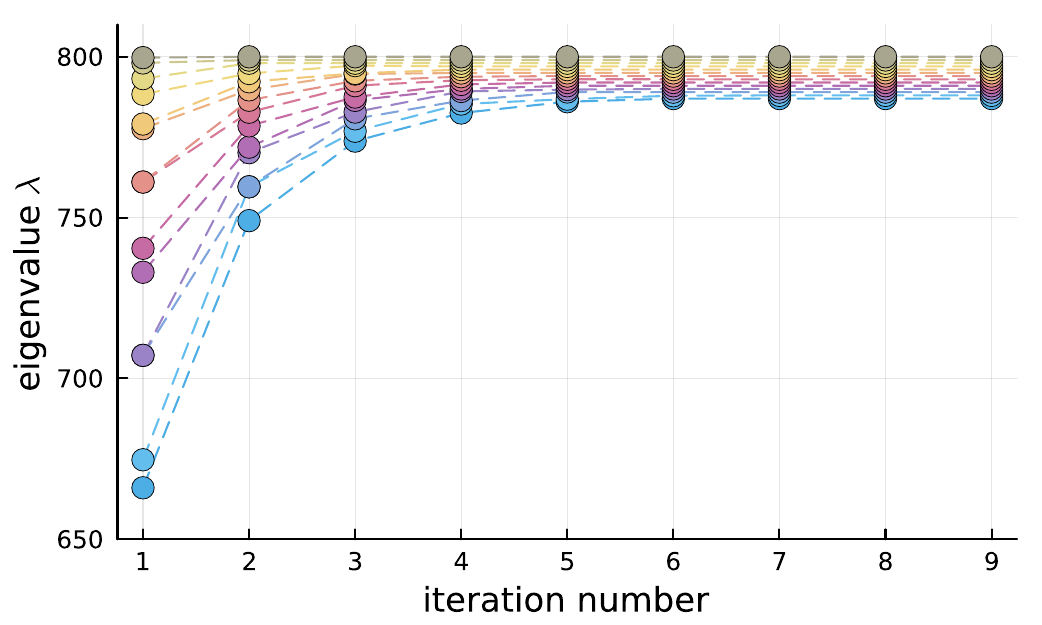}
        \caption{Eigenvalues largest modulus (LM)}
        \label{fig:eivaldiagLM}
    \end{subfigure}
    \hfill
    \begin{subfigure}[b]{0.45\textwidth}
        \centering
        \includegraphics[width=\textwidth]{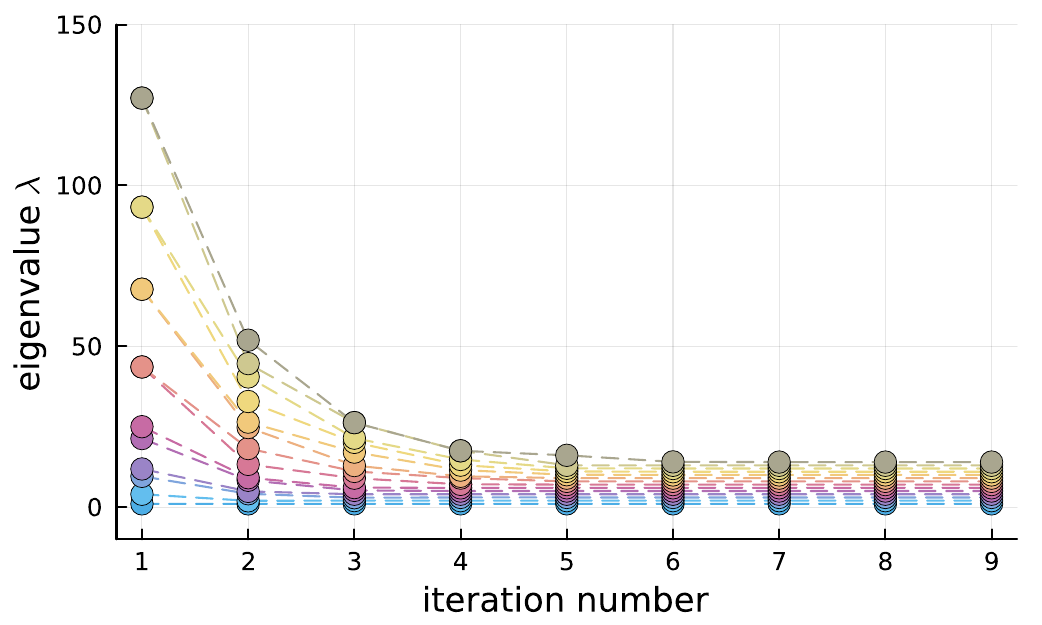}
        \caption{Eigenvalues smallest modulus (SM)}
        \label{fig:eivaldiagSM}
    \end{subfigure}
    \caption{rIRA to compute $\sa = 10$ Ritz pairs for the non-symmetric $800\times800$ toy matrix $\A$ with spectrum $\Lambda_{\A} = \{1,2,3,\dots,800\}$ and for two parts of the spectrum, namely largest and smallest modulus. \ref{fig:diagLM} and~\ref{fig:diagSM} : residual norms $\norm*{\A \rivec_i - \rival_i \rivec_i} / \norm*{ \rivec_i}$ along the iterations,  \ref{fig:eivaldiagLM} and~\ref{fig:eivaldiagSM} : eigenvalues $\rival_i$ along the iterations.}
    \label{fig:diag}
\end{figure}

\subsection{Orthogonalization process}
We continue by studying the condition number of the basis produced by the different orthogonalization processes introduced in section~\ref{sec:orthog}, notably RGS that solves the least squares problem in line~\ref{alg:sketch-ortho-LS} of Algorithm~\ref{alg:sketch-orthonormal} by computing the QR factorization of $(\Om \V)$ and rCGS / rCGS2 that considers that $\Om \V$ is orthogonal and solves the least squares problem using the transpose $(\Om \V)^T$. The condition number $\kappa(\V)$ is computed during the inner iterations of one Arnoldi factorization of dimension $\ba = 200$ for \href{https://sparse.tamu.edu/Grund/poli4}{poli4}. The results are presented in Figure~\ref{fig:Conds-hard}, where the condition number of the basis obtained using CGS and CGS2 is also displayed. For this example, both CGS and rCGS are not stable, after 50 iterations the condition number of the basis increases to $10^{18}$. Note that using the weak version of rCGS2 that relies on a cheaper reorthogonalization step (for more details see section~\ref{sec:orthog}) is not sufficient to produce a well conditioned basis.  As shown in more details in Figure~\ref{fig:Conds-hard-zoom}, rCGS2 and RGS maintain a condition number $\kappa(V) \approx 3$, which is to be expected from stable sketched methods with $\epsilon = 1/2$, given Corollary~\ref{cor:rando_sval}. 
Figures~\ref{fig:poli4-evp-RGS} and~\ref{fig:poli4-evp-rCGS2} display the $\sa = 100$ eigenvalues of largest modulus obtained using rIRA with either RGS or rCGS2 for each inner Arnoldi iteration. The results obtained using the eigs method are used as a reference, where eigs is an implementation of the IRA method for general non-symmetric matrices, see \href{https://arpack.julialinearalgebra.org/latest/index.html#Arpack.eigs-Tuple{Any}}{Julia's eigs documentation}. In this situation, it should be noted that both processes produce similar results for a tolerance of $\norm*{\A \rivec - \rival \rivec} / \norm{\rivec} \leq 10^{-8}$ for each approximated eigenpairs, and this tolerance was reached after two outer rIRA iterations. The fact that $\VOm \coloneqq \Om \V$ obtained from rCGS2 is better conditioned than the other randomized methods did not lead to faster convergence in the eigenvalue solver in our experiments so far. The unstable orthogonalization processes did not lead to convergence in the eigenvalue solver. 

\begin{figure}
    \begin{subfigure}[b]{0.45\textwidth}
        \centering
        \includegraphics[width =\textwidth]{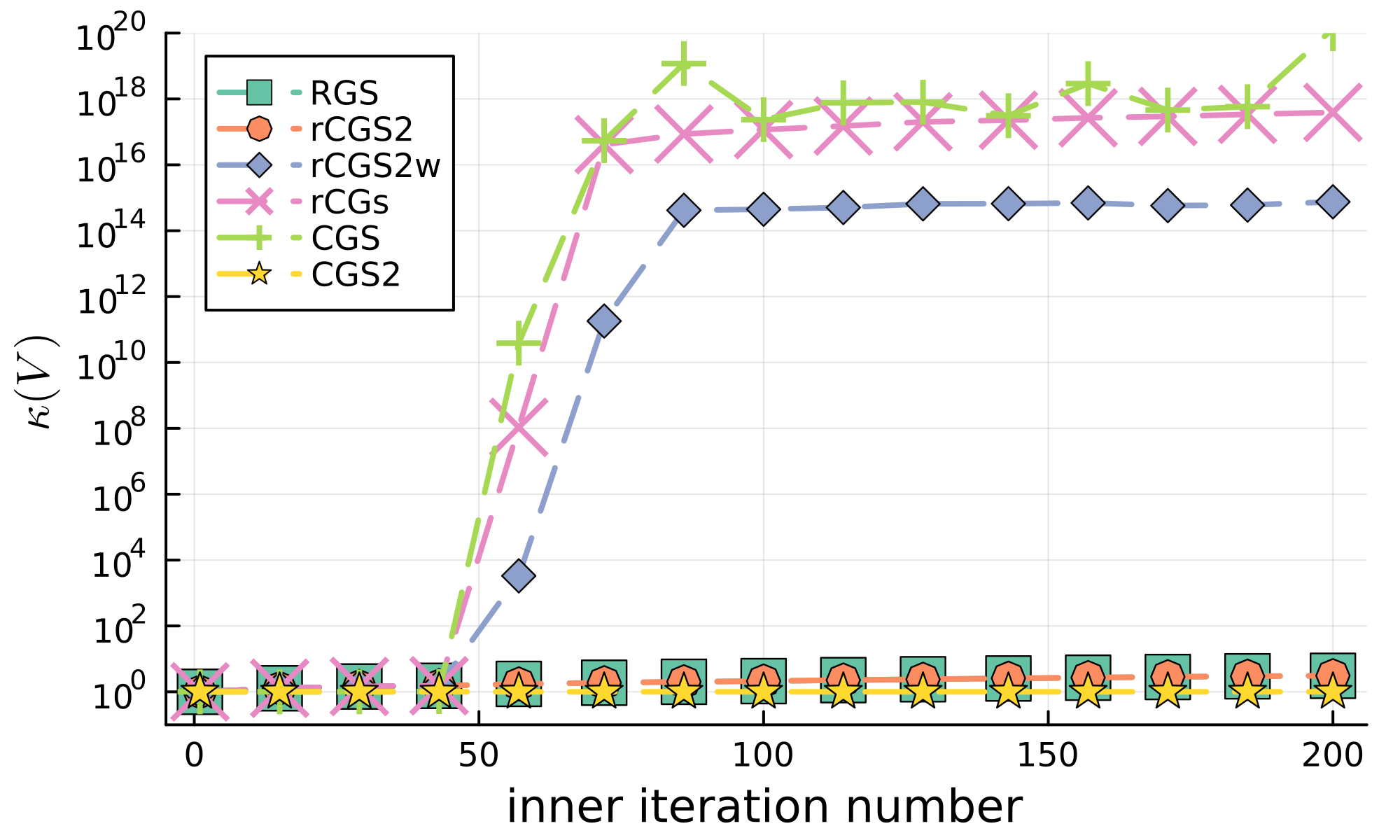}
        \caption{Condition number of $\V \in \Rnm{n}{(\# \text{ of inner iteration})}$}
        \label{fig:Conds-hard}
    \end{subfigure}
    \hfill
    \begin{subfigure}[b]{0.45\textwidth}
        \centering
        \includegraphics[width =\textwidth]{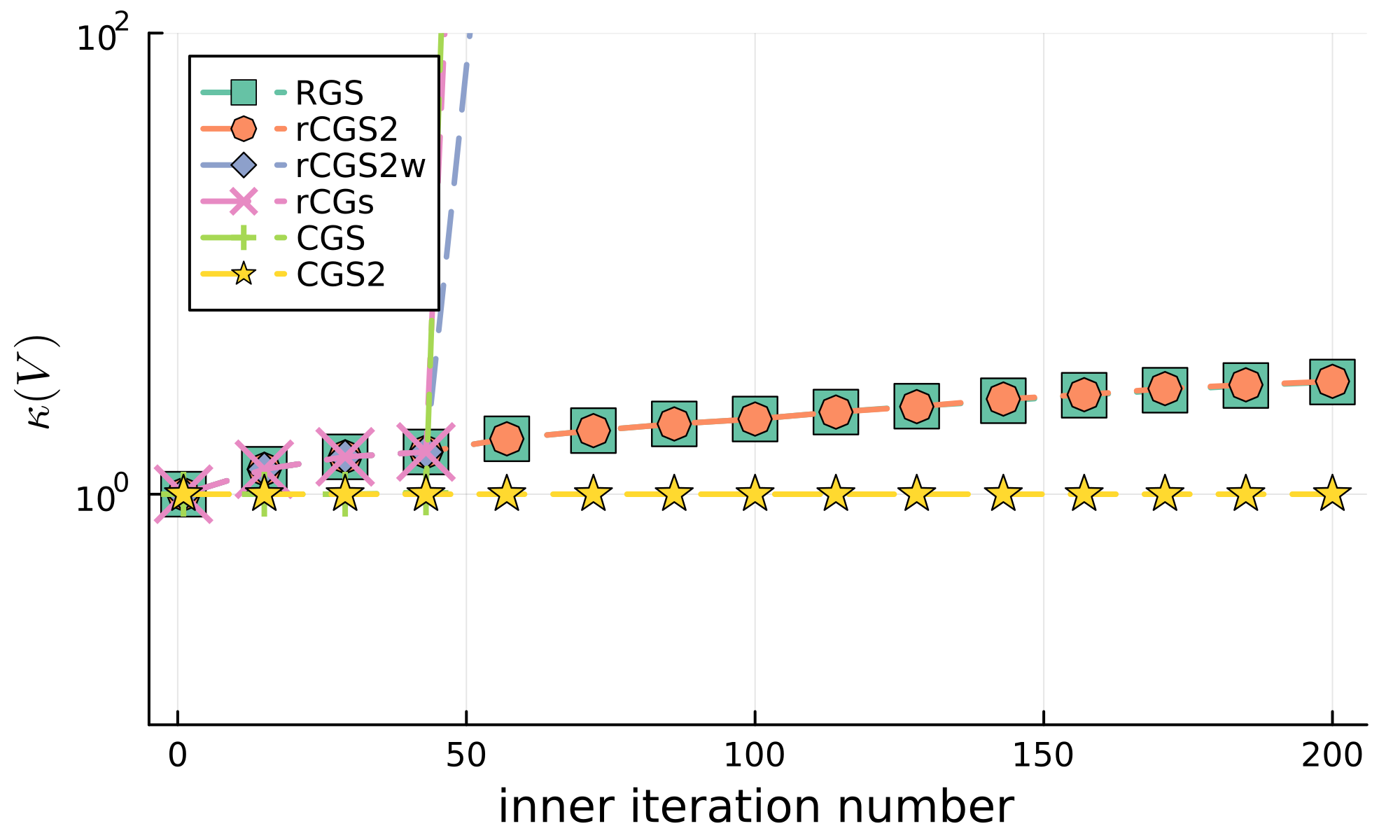}
        \caption{Zoom of Figure~\ref{fig:Conds-hard} on stable processes}
        \label{fig:Conds-hard-zoom}
    \end{subfigure}
    \hfill
    \begin{subfigure}[b]{0.45\textwidth}
        \centering
        \includegraphics[width =\textwidth]{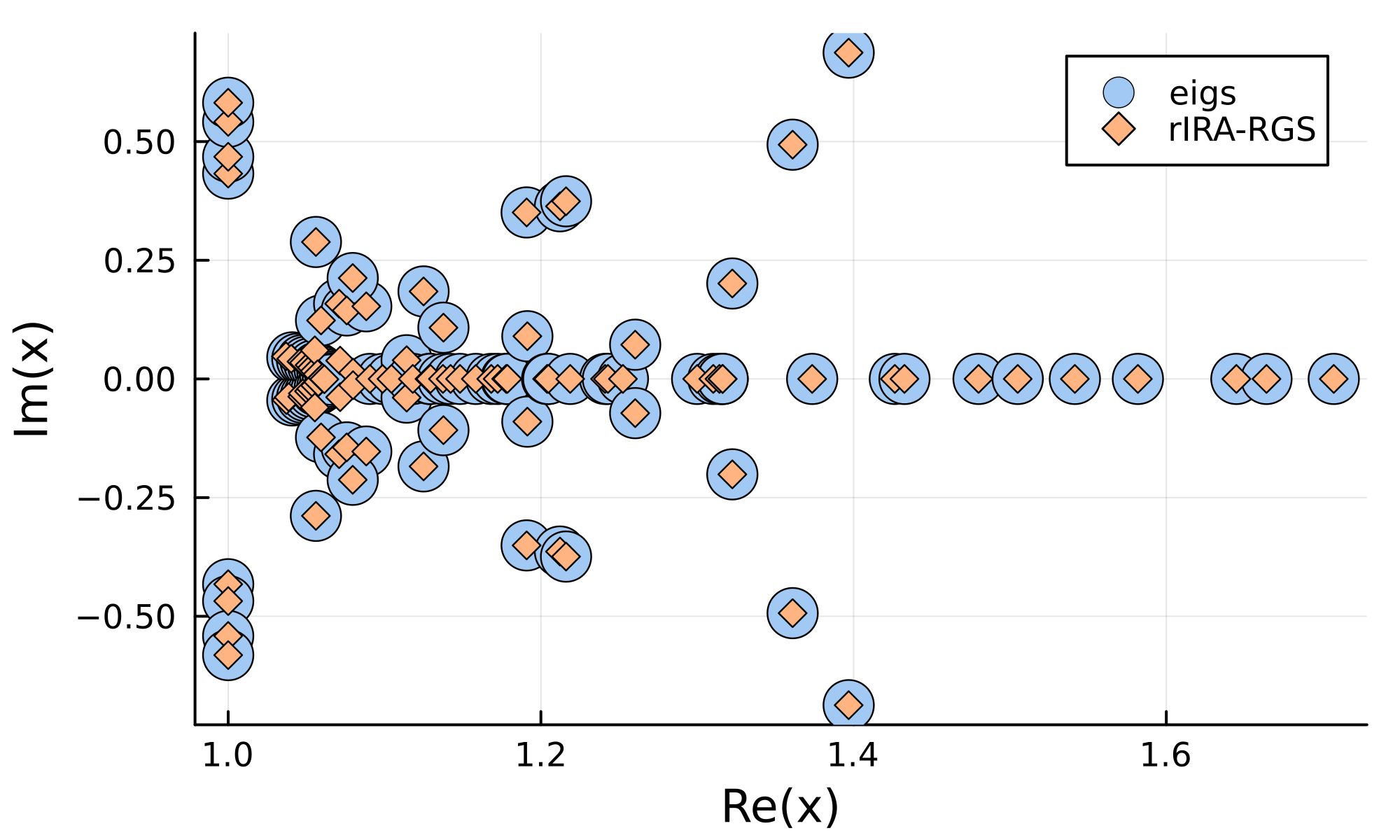}
        \caption{Eigenvalues using RGS for rIRA}
        \label{fig:poli4-evp-RGS}
    \end{subfigure}
    \hfill
    \begin{subfigure}[b]{0.45\textwidth}
        \centering
        \includegraphics[width =\textwidth]{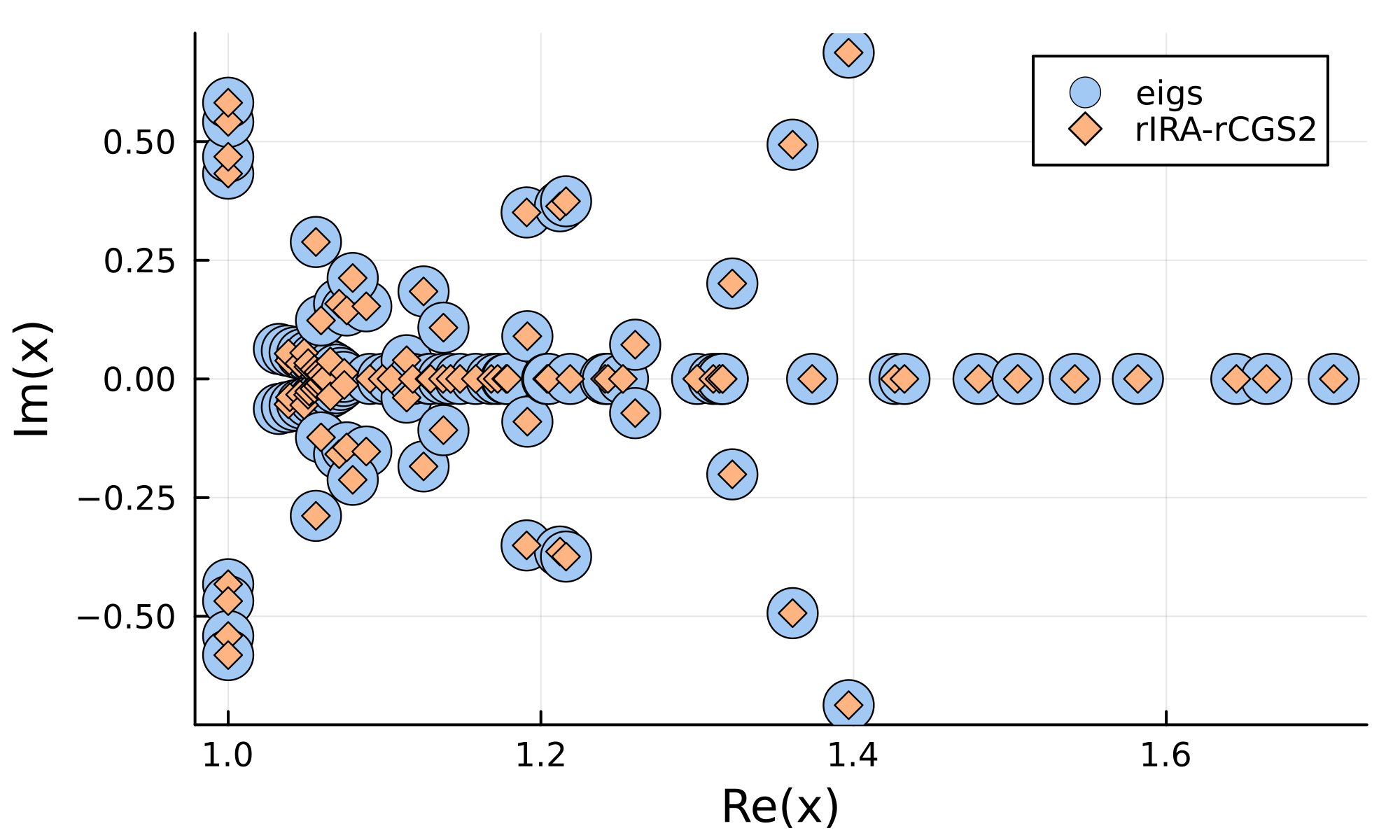}
        \caption{Eigenvalues using rCGS2 for rIRA}
        \label{fig:poli4-evp-rCGS2}
    \end{subfigure}
    \hfill
    \caption{Stability of orthogonalization processes and resulting eigenvalues for RGS and rCGS2 with $\A$ = poli4, $\ba$ = 200 and $\sa = 100$}
    \label{fig:Conds}
\end{figure}

\subsection{Comparison to other eigensolvers}
We compare in this section rIRA with two other eigensolvers, Julia's eigs solver and a randomized eigensolver introduced in \cite{Balabanov2021RandomizedblockGram}. We start with a comparison with Julia's eigs, an implementation of the IRA method for general non-symmetric matrices, see \href{https://arpack.julialinearalgebra.org/latest/index.html#Arpack.eigs-Tuple{Any}}{Julia's eigs documentation} and the ARPACK user guide \cite{Lehoucq1998ARPACKusersguide}. Figure~\ref{fig:Hamrle} shows the eigenpairs convergence and the eigenvalues for the non-symmetric matrix \href{https://sparse.tamu.edu/Hamrle/Hamrle3}{Hamrle3}. A total of $\sa = 200$ eigenpairs is sought in a Krylov subspace of dimension $\ba = 500$ and the tolerance is $\eta = 10^{-6}$. For the sake of clarity, only the maximum and minimum residual norms $\norm*{\A \rivec - \rival \rivec} / \norm*{\rivec}$ are displayed. The closeness between the maximum residual norm and its sketch shows that the $\epsilon$-embedding of the residual is maintained throughout the iterations. Note that the eigenvalues computed by rIRA correspond to the result of eigs. This is a rather difficult case, given that LM's eigenvalues are not well separated, as they all lie in the circle centered around 0 with a radius of 2 and therefore they have similar moduli. 

Another interesting example is the computation of the SM eigenvalues for the matrix \href{https://sparse.tamu.edu/VLSI/vas_stokes_1M}{vas\_stokes\_1M}. It is well known that using shift and invert with $\hat{\A} \coloneqq (\A - \sigma I)^{-1}$ for some $\sigma \approx \eival_{max}$ is a relevant method for computing SM eigenvalues, since it  benefits from the better convergence of eigensolvers when seeking LM eigenvalues, see \cite{Parlett1987Complexshiftinvert} for example. However, for large matrices it might not be feasible to compute their LU decomposition to invert the linear system. In this case, rIRA's capability of focusing on a specific part of the spectrum can be exploited for computing SM eigenvalues. The $\sa = 50$ SM eigenvalues of vas\_stokes\_1M and their convergence are shown in Figure~\ref{fig:vas} for a Krylov subspace of dimension $\ba = 100$. The tolerance of $\eta = 10^{-10}$ is reached after more than 800 iterations. However, a standard eigenvalue solver that uses only preconditioning to recover SM eigenvalues might not converge, see the comparison with randeigs below. 

\begin{figure}
    \begin{subfigure}[b]{0.45\textwidth}
        \centering
        \includegraphics[width =\textwidth]{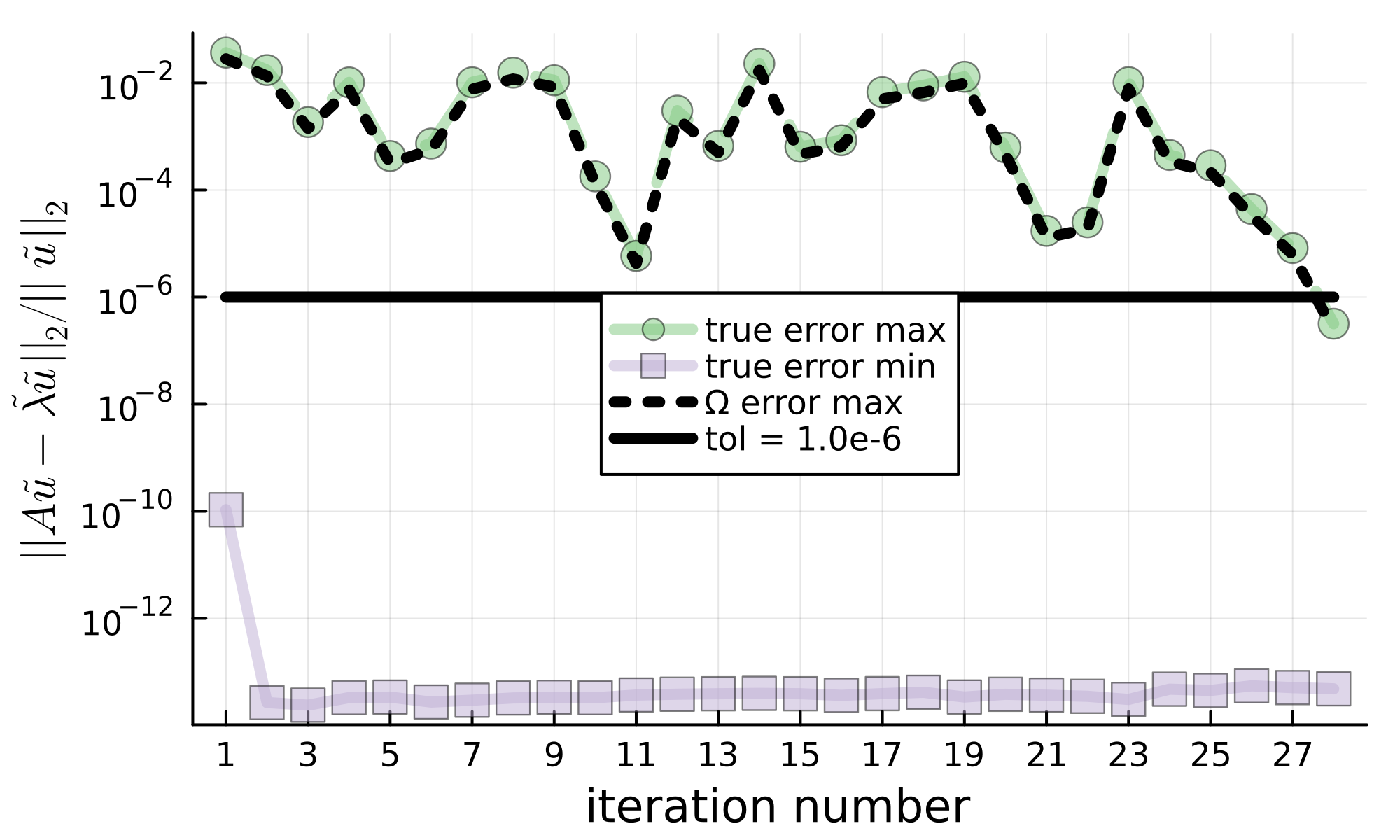}
        \caption{residual errors}
        \label{fig:hamrle-res}
    \end{subfigure}
    \hfill
    \begin{subfigure}[b]{0.45\textwidth}
        \centering
        \includegraphics[width =\textwidth]{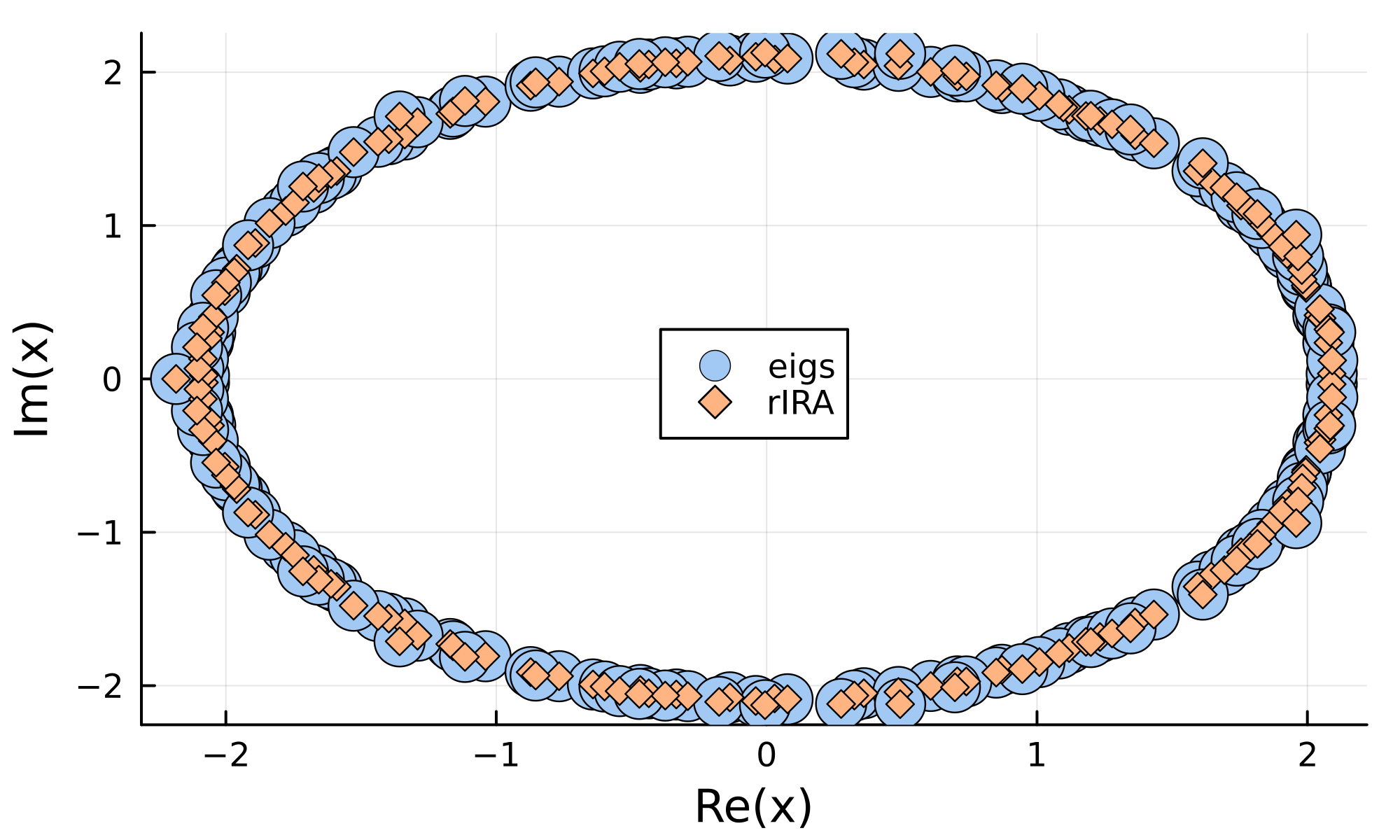}
        \caption{eigenvalues}
        \label{fig:hamrle-evp}
    \end{subfigure}
    \caption{$\A$ = Hamrle3 of size $n \approx 1.4 \times 10^6$, $\sa = 200$ LM eigenpairs computed in a Krylov dimension $\ba = 500$ }
    \label{fig:Hamrle}
\end{figure}

\begin{figure}
    \begin{subfigure}[b]{0.45\textwidth}
        \centering
        \includegraphics[width =\textwidth]{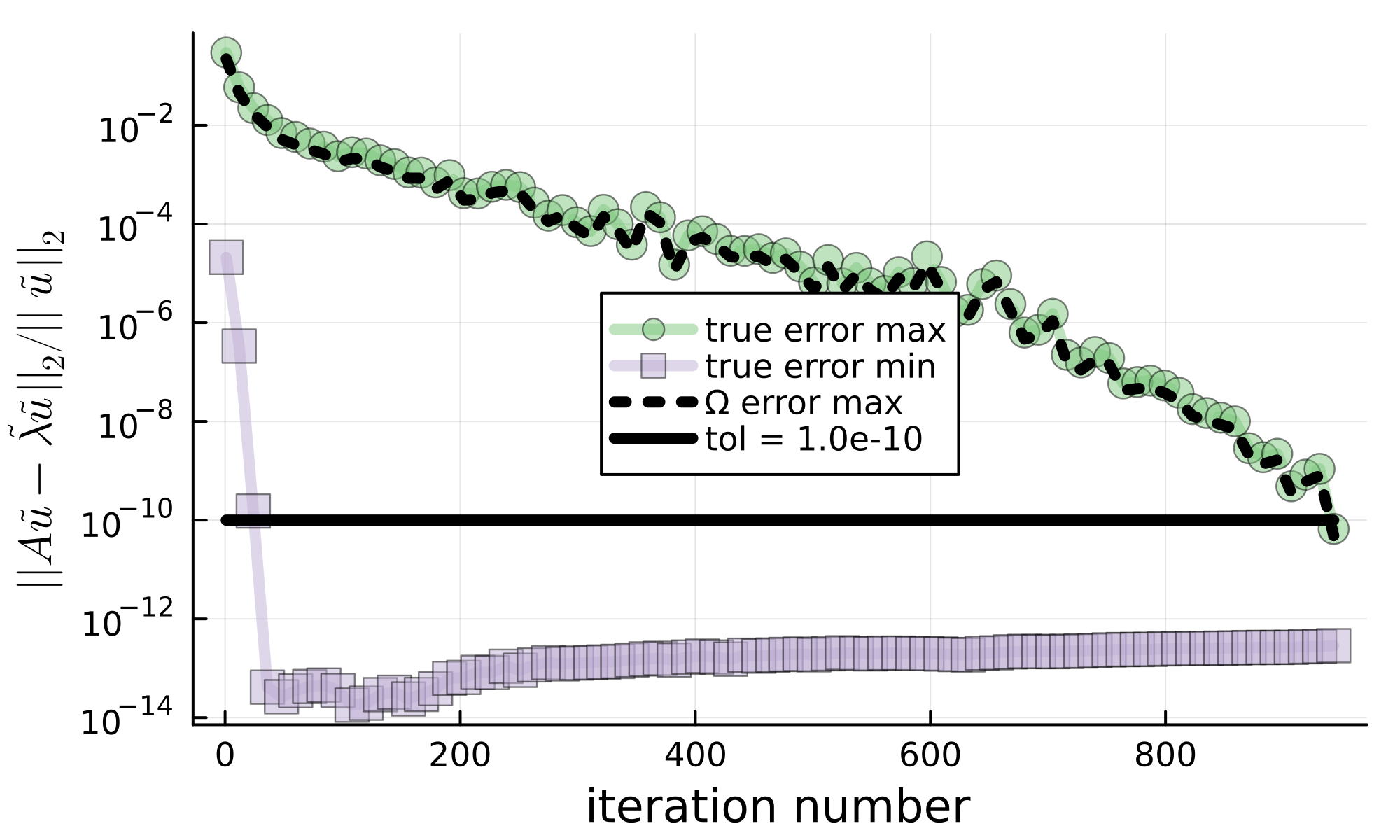}
        \caption{residual errors}
        \label{fig:vas-res}
    \end{subfigure}
    \hfill
    \begin{subfigure}[b]{0.45\textwidth}
        \centering
        \includegraphics[width =\textwidth]{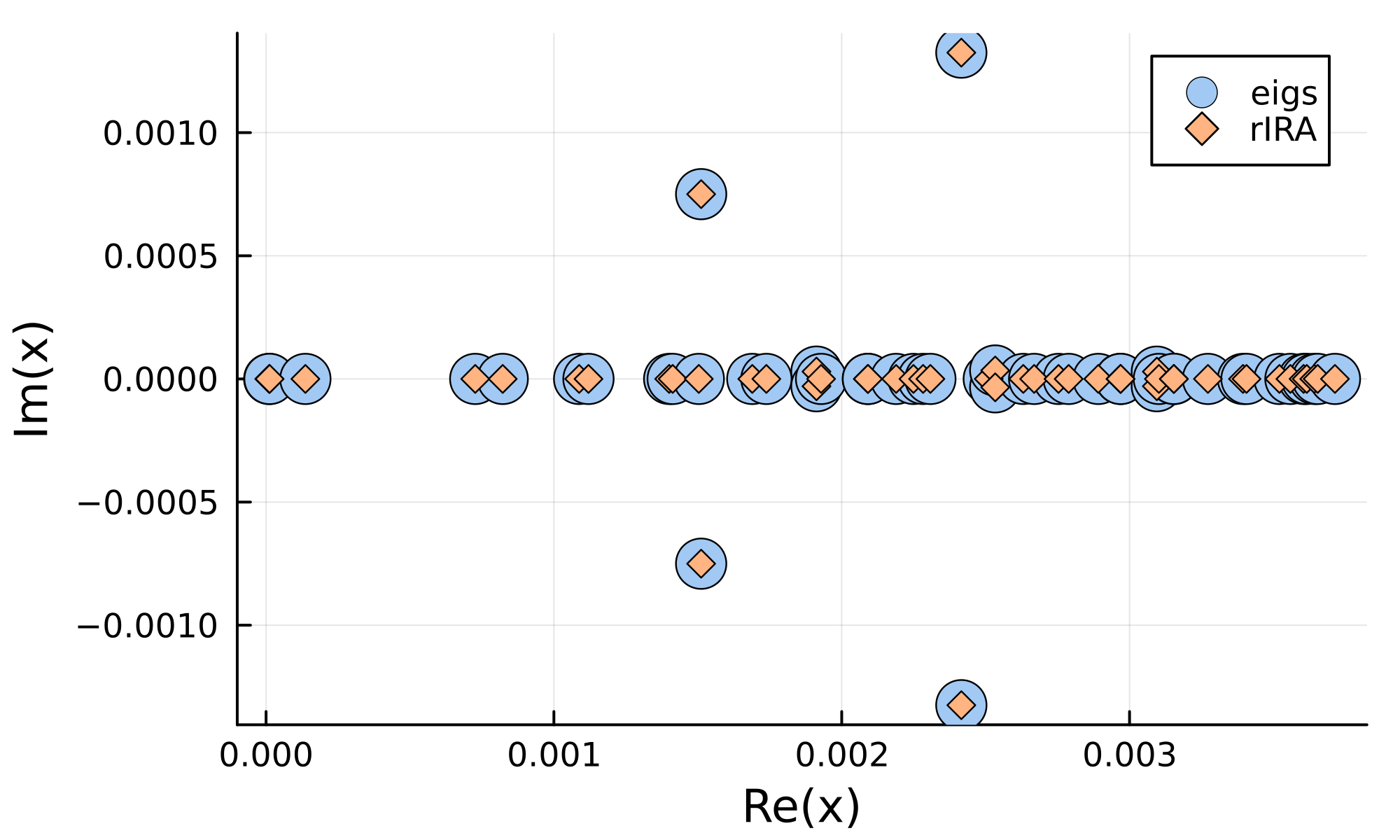}
        \caption{eigenvalues}
        \label{fig:vas-evp}
    \end{subfigure}
    \caption{$\A$ = vas\_stokes\_1M of size $n \approx 1.1 \times 10^6$, $\sa = 50$ SM eigenpairs computed in a Krylov dimension $\ba = 100$ }
    \label{fig:vas}
\end{figure}

Table~\ref{tab:eigscomp} displays results obtained for various matrices and includes runtimes obtained on a node with 2x Cascade Lake Intel Xeon 5218 16 cores, 2.4GHz processor and 192GB of RAM, as specified in \href{https://paris-cluster-2019.gitlabpages.inria.fr/cleps/cleps-userguide/architecture/architecture.html#cleps-compute-nodes}{this documentation, nodes[01-20]}. Here rIRA uses the RGS algorithm for the orthogonalization process and sets as convergence criterion $\eta = 10^{-10}$.  The table indicates the number Nit of outer iterations required to reach this accuracy by eigs and rIRA and the number MVp of matrix-vector products with $\A$.  A / indicates that the accuracy was not reached within 1000 iterations. We run different experiments for different values of $\sa$ and $m$, and for two subsets of interest which are Largest Modulus or Smallest Modulus.
We note that there are several instances when rIRA is faster than eigs, although our code is sequential and not optimized.  Larger speedups are expected in parallel, since randomization and in particular RGS are known to reduce communication costs while maintaining stability.
In many cases there is a slight difference in the number of iteration between the two methods. Also, rIRA can converge in less iteration than IRA and still do more matrix-vector products. This can be explained by the fact that the two codes differ in many implementation details, for instance rIRA does not yet use deflation or other techniques to improve convergence. We emphasize that, in case of convergence, rIRA reaches the same accuracy of the residual norms of the eigenpairs as IRA for the matrices in our test set. In addition, there is no case where IRA converged and rIRA did not.  In almost all cases, rIRA is faster than IRA.

\begin{table}
    \centering
    \begin{tabular}{|c | c  |c| c| c| c| c| }
        \hline
        \textbf{k} & \textbf{m} & \textbf{Which} & \textbf{Method} & \textbf{Nit} & \textbf{MVp} & \textbf{Time (s)} \\    
        \hline    
        \hline
        \rowcolor{white} \multicolumn{7}{|c|}{\textbf{\href{https://sparse.tamu.edu/CEMW/tmt_unsym}{tmt\_unsym}} (Electromagnetic Problem) of size $n = 9 \times 10^5$ with $4.5 \times 10^6$ nonzeros.} \\
        \hline
        20 & 200 & LM & eigs & 98 &	13446 & 1417 \\
        \hline
        20 & 200 & LM & rIRA & 102 &	15350 &	\textbf{1296} \\
        \hline
        20 & 200 & SM & eigs & 74	& 12995	&  1149 \\
        \hline
        20 & 200 & SM & rIRA & 69 & 12440 & \textbf{1054} \\
        \hline
        20 & 100 & SM & eigs & 201	& 15360	&  929 \\
        \hline
        20 & 100 & SM & rIRA & 187 & 14980 & \textbf{855} \\
        \hline
        \hline
        \rowcolor{white} \multicolumn{7}{|c|}{\textbf{\href{https://sparse.tamu.edu/VLSI/vas_stokes_1M}{Vas\_stokes\_1M}}  (Semiconductor Process Problem) of size $n = 1.1 \times 10^6$ with $3.5 \times 10^7$ nonzeros.} \\
        \hline
        50 & 200 & LM & eigs & 7 & 985 & 	175  \\
        \hline
        50 & 200 & LM & rIRA & 7 & 1100 & \textbf{161}  \\
        \hline
        50 & 200 & SM & eigs & 336 & 40880 & 7727  \\
        \hline
        50 & 200 & SM & rIRA & 575 & 86300 & 12860  \\
        \hline
        \hline
        \rowcolor{white} \multicolumn{7}{|c|}{\textbf{\href{https://sparse.tamu.edu/Hamrle/Hamrle3}{Hamrle3}} (Circuit Simulation Problem) of size $n = 1.4 \times 10^6$ with $5.5 \times 10^6$ nonzeros.} \\
        \hline
        50 & 200 & LM & eigs & 19 & 2377 & 268  \\
        \hline
        50 & 200 & LM & rIRA & 18 & 2750 & 336  \\
        \hline
        50 & 200 & SM & eigs & / & / & /  \\
        \hline
        50 & 200 & SM & rIRA & / & / & /  \\
        \hline
        \hline
        \rowcolor{white} \multicolumn{7}{|c|}{\textbf{\href{https://sparse.tamu.edu/Bourchtein/atmosmodl}{atmosmodl}} (Computational Fluid Dynamics) of size $n = 1.5 \times 10^6$ with $1.0 \times 10^7$ nonzeros.} \\
        \hline
        50 & 200 & LM & eigs & 30 &	3853 & 725 \\
        \hline
        50 & 200 & LM & rIRA & 32 &	4850 &	\textbf{638} \\
        \hline
        50 & 200 & SM & eigs & 32 & 4124	&  740 \\
        \hline
        50 & 200 & SM & rIRA & 34 & 5150 & \textbf{687} \\
        \hline
        \hline
        \rowcolor{white} \multicolumn{7}{|c|}{\textbf{\href{https://sparse.tamu.edu/Janna/ML_Geer}{ML\_Geer}} (Strucural Problem)  of size $n = 1.5 \times 10^6$ with $1.1 \times 10^8$ nonzeros.} \\
        \hline
        50 & 200 & LM & eigs & 3 & 446 & 142  \\
        \hline
        50 & 200 & LM & rIRA & 3 & 500 & \textbf{134}  \\
        \hline
        50 & 200 & SM & eigs & / & / & /  \\
        \hline
        50 & 200 & SM & rIRA & / & / & /  \\
        \hline
        
    \end{tabular}
    \caption{Comparison in terms of number of iterations (Nit), Matrix-vector products (Mvp) and runtime between rIRA and eigs for an error tolerance $\eta = 10^{-10}$.}
    \label{tab:eigscomp}    
\end{table}

We end this section by comparing rIRA with another randomized eigensolver introduced initially in a simpler version in \cite{Balabanov2021RandomizedblockGram}, for which the Matlab code is available \href{https://github.com/obalabanov/randKrylov}{on GitHub (link)}. It is called randeigs and has some similarities with eigs, while restarting using a computed Schur decomposition of $\Hba$. However it does not use the shifted QR algorithm to update the factorization and has to fully restart by doing $m$ inner iterations on each outer iteration, compared to rIRA that does only $\nshi$ thanks to implicit restarting. More details on randeigs are given in its documentation. To compute the SM eigenvalues, it either uses shift and invert through an LU decomposition of $\A$ or a simpler shift strategy $\hat{\A} = I_n - A / (\eival_{max} + 0.1)$ when the LU is too difficult to compute, as for the vas\_stokes\_1M matrix discussed previously.  The experiments are collected in Table~\ref{tab:randeigscomp} and are carried out on a MacBook Pro with an M1 processor and 16 GB RAM. The main conclusions are that rIRA outperforms randeigs in LM cases and also in SM cases with the simple shift strategy, as it can be seen on matrices poli4 and lung2 where rIRA is the best method overall with up 3x speedup in both LM and SM cases, even with twice the number of iterations compared to randeigs.  This highlights the usefulness of the SM focus capacity of rIRA when no factorization of the input matrix is available. However when a stable LU factorization is available as with tmt\_unsym, shift and invert may outperform the inherent capacity of rIRA to focus on the smallest eigenvalues as shown with the tmt\_unsym result. Note that shift and invert can be implemented within rIRA to mitigate this situation, and then it will be possible to choose between preconditioning or the SM shifts strategy of rIRA.  

\begin{table}
    \centering
    \begin{tabular}{|c | c  |c| c| c| c|}
        \hline
        \textbf{k} & \textbf{m} & \textbf{Which} & \textbf{Method} & \textbf{Nit} & \textbf{Time (s)} \\    
        \hline
        \hline
        \multicolumn{6}{|c|}{\textbf{\textbf{\href{https://sparse.tamu.edu/Grund/poli4}{poli4}}} (Computational Fluid Dynamics) of size $n = 3.3 \times 10^4$ with $7.3 \times 10^4$ nonzeros.} \\
        \hline
        100 & 150 & SM & randeigs (shift \& invert) & 12 &  10.0 \\
        \hline
        100 & 150 & SM & randeigs (simple shift) & 16	&  32.2 \\
        \hline
        100 & 150 & SM & rIRA & 31 & \textbf{8.9} \\
        \hline    
        \hline
        \multicolumn{6}{|c|}{\textbf{\textbf{\href{https://sparse.tamu.edu/Norris/lung2}{lung2}}} (Computational Fluid Dynamics) of size $n = 1.1 \times 10^5$ with $5.0 \times 10^5$ nonzeros.} \\
        \hline
        10 & 100 & LM & randeigs & 6  & 7.1 \\
        \hline
        10 & 100 & LM & rIRA  & 6 & \textbf{2.1} \\
        \hline
        10 & 100 & SM & randeigs (shift \& invert) & 16	&  19.4 \\
        \hline
        10 & 100 & SM & randeigs (simple shift) & 6    &  14.9 \\
        \hline
        10 & 100 & SM & rIRA & 15	&  \textbf{5.9} \\
        \hline
        \hline
        \multicolumn{6}{|c|}{\textbf{\textbf{\href{https://sparse.tamu.edu/CEMW/tmt_unsym}{tmt\_unsym}}} (Electromagnetics Problem) of size $n = 9 \times 10^5$ with $4.5 \times 10^6$ nonzeros.} \\
        \hline
        10 & 100 & LM & randeigs & 24 & 209.7 \\
        \hline
        10 & 100 & LM & rIRA  & 44 & \textbf{108.9} \\
        \hline
        10 & 100 & SM & randeigs (shift \& invert) & 1 & 30.4 \\
        \hline
        10 & 100 & SM & randeigs (simple shift)	 &	 20 & 396.7 \\
        \hline
        10 & 100 & SM & rRIRA 	& 39 &  92.4 \\
        \hline
    \end{tabular}
    \caption{Comparison in terms of number of iterations (Nit), Matrix-vector products (Mvp) and runtime between rIRA and randeigs for an error tolerance $\eta = 10^{-10}$, or $\eta = 10^{-4}$ for tmt\_unsym.}
    \label{tab:randeigscomp}    
\end{table}

\section*{Acknowledgements}
This project has received funding from the European Research
Council (ERC) under the European Union’s Horizon 2020 research and innovation program (grant agreement No 810367).

\printbibliography

\end{document}